\documentclass[11pt,reqno,a4paper]{amsart}
\usepackage{amsmath,amssymb}
\usepackage{amsrefs}
\usepackage{mathtools}
\usepackage{enumerate}
\usepackage[hidelinks]{hyperref}

\newtheorem{theorem}{Theorem}[section]
\newtheorem{remark}[theorem]{Remark}
\newtheorem{proposition}[theorem]{Proposition}
\newtheorem{lemma}[theorem]{Lemma}
\newtheorem{corollary}[theorem]{Corollary}

\numberwithin{equation}{section}

\makeatletter
\renewcommand*\env@matrix[1][\arraystretch]{%
  \edef\arraystretch{#1}%
  \hskip -\arraycolsep
  \let\@ifnextchar\new@ifnextchar
  \array{*\c@MaxMatrixCols c}}
\makeatother

\DeclarePairedDelimiter\abs{\lvert}{\rvert}
\DeclarePairedDelimiter\norm{\lVert}{\rVert}

\begin{document}
 
\title[Solvability for the linearized Ginzburg-Landau equation]{Solvability for the Ginzburg-Landau equation linearized at the degree-one vortex}

\author[M.~del Pino]{Manuel del Pino}
\address{\noindent M.~del Pino: Department of Mathematical Sciences University of Bath, Bath BA2 7AY, United Kingdom}
\email{m.delpino@bath.ac.uk}

\author[R.~Juneman]{Rowan Juneman}
\address{\noindent R.~Juneman: Department of Mathematical Sciences University of Bath, Bath BA2 7AY, United Kingdom}
\email{rj493@bath.ac.uk}

\author[M.~Musso]{Monica Musso}
\address{\noindent M.~Musso: Department of Mathematical Sciences University of Bath, Bath BA2 7AY, United Kingdom.}
\email{m.musso@bath.ac.uk}

\begin{abstract}
We consider the Ginzburg-Landau equation in the plane linearized around the standard degree-one vortex solution \(W(x)=w(r)e^{i\theta}\). Using explicit representation formulae for the Fourier modes in \(\theta\), we obtain sharp estimates for the inverse of the linearized operator which hold for a large class of right-hand sides. This theory can be applied, for example, to estimate the inverse after dropping the usual orthogonality conditions. 
\end{abstract}

\maketitle

\section{Introduction}

The complex-valued Ginzburg-Landau equation in the plane
\begin{equation}
\label{GLeq}
    \Delta u+(1-\abs{u}^{2})u=0,\quad\text{in }\mathbb{R}^{2},
\end{equation}admits a well-known, symmetric degree-one vortex solution of the form 
\begin{equation}
W(x)=w(r)e^{i\theta},
\end{equation}
where \(x_{1}=r\cos(\theta)\) and \(x_{2}=r\sin(\theta)\) denote the usual polar coordinates in \(\mathbb{R}^{2}\), and \(w=w(r)\) denotes the unique solution of the problem
\begin{equation}
\label{wODE}
\begin{gathered}
w''+\frac{1}{r}w'-\frac{1}{r^{2}}w+(1-w^2)w=0,\quad\text{in }(0,\infty),\\
    w(0)=0,\quad w(r)\to1\text{ as }r\to\infty.
\end{gathered}    
\end{equation}Referring to \cites{chenelliottqi1994,herveherve1994} for the key properties of \(w(r)\), it is known that \(0<w(r)<1\), \(w'(r)>0\) for all \(r>0\), and \(w(r)\) has the asymptotic behaviour
\begin{align}
\begin{split}
    w(r)&=\alpha r+O(r^{3})\text{ as }r\to0^{+},\text{ for some }\alpha>0,\label{wproperties}\\
    w(r)&=1-1/(2r^2)+O(1/r^4)\text{ as }r\to\infty,\\
    w'(r)&=1/r^{3}+O(1/r^{5})\text{ as }r\to\infty.
\end{split}    
\end{align}

As shown by Mironescu \cite{mironescu1996}, the symmetric degree-one vortex \(W(x)\) plays a distinguished role in the classification of solutions to (\ref{GLeq}), in the sense that any solution of (\ref{GLeq}) satisfying \(\int_{\mathbb{R}^{2}}(1-\abs{u}^{2})^{2}\:\:dx<\infty\) and having degree one at infinity must necessarily be equal to \(e^{i\gamma}W(x-x_{0})\) for some \(\gamma\in\mathbb{R}\), \(x_{0}\in\mathbb{R}^{2}\). In view of this fact, there has long been interest in understanding the stability and nondegeneracy of the degree-one vortex, and in developing a suitable solvability theory for the linearized operator
\begin{equation}
\label{linearizedop}
L[\phi]:=\Delta\phi+(1-\abs{W}^{2})\phi-2\operatorname{Re}(\overline{W}\phi)W
\end{equation}of (\ref{GLeq}) around \(W\). 

Associated to (\ref{linearizedop}) are three elements of the kernel 
\begin{equation}
L\left[iW\right]=L\left[\frac{\partial W}{\partial x_{1}}\right]=L\left[\frac{\partial W}{\partial x_{2}}\right]=0,
\end{equation}reflecting the invariance of equation (\ref{GLeq}) under the symmetry groups of translations and phase shifts. These functions are known to span the entire kernel in \(L^{\infty}\), see \cite{pacardriviere2000}*{Theorem 3.2}.

Another important property of (\ref{linearizedop}) is the non-negativity of the corresponding quadratic form 
\begin{equation}
\label{quadform}
B(\phi,\phi):=\int_{\mathbb{R}^{2}}\abs{\nabla\phi}^{2}-(1-\abs{W}^{2})\abs{\phi}^{2}+2(\operatorname{Re}(\overline{W}\phi))^{2}\:\:dx,    
\end{equation}a fact established in \cite{delpinofelmerkowalczyk2004} for perturbations \(\phi\) belonging to the natural Hilbert space
\begin{equation}
    H:=\left\{\phi\in H^{1}_{\text{loc}}(\mathbb{R}^{2}):\norm{\phi}_{H}^{2}:=\int_{\mathbb{R}^{2}}\abs{\nabla\phi}^{2}+(1-\abs{W}^{2})\abs{\phi}^{2}+(\operatorname{Re}(\overline{W}\phi))^{2}\:\:dx<\infty\right\}.
\end{equation}In precise terms, it is shown in \cite{delpinofelmerkowalczyk2004} that \(B(\phi,\phi)\geq0\) for all \(\phi\in H\), with \(B(\phi,\phi)=0\) in \(H\) if and only if
\begin{equation}
\label{philinearcomb}
    \phi=c_{1}\frac{\partial W}{\partial x_{1}}+c_{2}\frac{\partial W}{\partial x_{2}},\quad\text{for some}\quad c_{1},c_{2}\in\mathbb{R}, 
\end{equation}the possibility \(\phi=iW\) being excluded in this case as \(iW\notin H\). See also \cites{mironescu1995,mironescu1996,ovchinnikovsigal1997} for related results in \(L^{2}\) and for compactly supported perturbations.

A key corollary of \(B(\phi,\phi)\geq0\) is a Fredholm alternative \cite{delpinofelmerkowalczyk2004} for the problem
\begin{equation}
\label{linearizedeq}
    L[\phi]=h,\quad\text{in }\mathbb{R}^{2},
\end{equation}which asserts the existence of solutions to (\ref{linearizedeq}) for right-hand sides \(h\) satisfying 
\begin{equation}
\label{hintegrability}
    \int_{\mathbb{R}^{2}}\abs{h}^{2}(1+r^{2+\sigma})\:\:dx<\infty,\quad\text{for some }\sigma>0,
\end{equation}and the orthogonality conditions
\begin{equation}
\label{orthogconds}
    \langle h,iW \rangle=\left\langle h,\frac{\partial W}{\partial x_{1}} \right\rangle=\left\langle h,\frac{\partial W}{\partial x_{2}} \right\rangle=0,\quad\quad\langle u,v \rangle:=\operatorname{Re}\int_{\mathbb{R}^{2}}u\overline{v}\:\:dx.
\end{equation}Whilst this result provides a natural application of the preceding theory, it is relevant to note that assumption (\ref{hintegrability}) is rather restrictive in practise. This integrability condition does not hold, for example, for any of the three generators \(iW,\:\partial W/\partial x_{1},\:\partial W/\partial x_{2}\). Moreover, the proof of the Fredholm alternative, based on variational methods, does not provide any information on the precise decay rate of the solution.

The aim of this paper is to revisit the invertibility theory for (\ref{linearizedeq}) using a different approach, which is not limited by assumptions (\ref{hintegrability}) and (\ref{orthogconds}). 
Our motivation for returning to this problem comes partly from our study of the Gross-Pitaevskii equation
\begin{equation}
\label{gpeq}
    iu_{t}+\Delta u+(1-\abs{u}^{2})u=0,\quad(x,t)\in\mathbb{R}^{2}\times[0,T],
\end{equation}whose analysis using singular perturbation techniques requires one to deal with (\ref{linearizedeq}) where \(h\) is merely bounded, or worse, where \(h\) grows as \(r\to\infty\). In these situations the variational method cannot be applied, and one is forced to drop one or more of the orthogonality conditions (\ref{orthogconds}) as the relevant products are not integrable in the whole of \(\mathbb{R}^{2}\).

An equivalent form of equation (\ref{linearizedeq}), which is often easier to study, can be expressed in terms of the function \(\psi\) defined via \(\phi=iW\psi\). Writing \(\psi=\psi_{1}+i\psi_{2}\) for the real and imaginary parts of \(\psi\), and \(h=iW\big(h_{1}+ih_{2}\big)\), problem (\ref{linearizedeq}) reads
\begin{equation}
\label{linearizedeqpsi}
\Delta\psi+2\frac{w'(r)}{w(r)}\partial_{r}\psi+\frac{2i}{r^{2}}\partial_{\theta}\psi-2iw(r)^{2}\psi_{2}=h_{1}+ih_{2}.
\end{equation}
In the construction of multi-vortex solutions to (\ref{gpeq}), linearization around a first approximation built from a product of vortices naturally leads one to (\ref{linearizedeqpsi}) with right-hand sides of the form \(h_{1}=O(r^{-2})\), \(h_{2}=O(1)\) as \(r\to\infty\), and with \(h=iW\big(h_{1}+ih_{2}\big)\) bounded on compact sets. We intend to study this scenario here as a model for more general problems. Far from the origin the operator on the left-hand side of (\ref{linearizedeqpsi}) behaves roughly like \(\Delta\) in the real part, and like \(\Delta-2\) in the imaginary part. Accordingly, we expect the existence of a solution with \(\psi_{1}=O((\log r)^{2})\), \(\psi_{2}=O(1)\) as \(r\to\infty\). Consistent with these observations, it is convenient to define the norms
\begin{align}
    \norm{\phi}_{*}&:=\sup_{r\leq2}\:\abs{\psi}+\sup_{r\geq2}\:\abs{(\log r)^{-2}\psi_{1}}+\sup_{r\geq2}\:\abs{\psi_{2}},\label{normphi}\\
    \norm{h}_{**}&:=\sup_{r\leq2}\:\abs{h}+\sup_{r\geq2}\:\abs{r^{2}h_{1}}+\sup_{r\geq2}\:\abs{h_{2}},\label{normh}
\end{align}where \(\phi=iW\big(\psi_{1}+i\psi_{2}\big)\) and \(h=iW\big(h_{1}+ih_{2}\big)\). (Note that, since \(\abs{W}=w\sim r\) as \(r\to0^{+}\), a function \(h\) satisfying \(\norm{h}_{**}<\infty\) will typically have \(h_{1},h_{2}=O(r^{-1})\) as \(r\to0^{+}\)). 

Our main result reads as follows.

\begin{theorem}
\label{mainthm}
There exists a constant \(C>0\) such that for any \(h\) satisfying \(\norm{h}_{**}<\infty\) and the orthogonality conditions \(\left\langle h,\frac{\partial W}{\partial x_{1}} \right\rangle=\left\langle h,\frac{\partial W}{\partial x_{2}} \right\rangle=0\), equation (\ref{linearizedeq}) has a solution \(\phi\) such that
\begin{equation}
\label{mainestimate}
    \norm{\phi}_{*}\leq C\norm{h}_{**}.
\end{equation}
\end{theorem}

\begin{remark}
\label{noorthogconds}
In the above result, no reference is made to the orthogonality condition \(\langle h,iW \rangle=0\) as \(h\overline{iW}\) is not necessarily integrable in \(\mathbb{R}^{2}\). If \(h\) does not satisfy any of the orthogonality conditions (\ref{orthogconds}), our method shows that a solution of (\ref{linearizedeq}) still exists with the weaker estimate
\begin{equation}
    \abs{\psi_{1}}\leq Cr\norm{h}_{**}, \quad\quad\abs{\psi_{2}}\leq C\norm{h}_{**},
\end{equation}    
for \(r\geq2\).
\end{remark}

The proof of Theorem \ref{mainthm} relies crucially on a delicate ODE analysis related to the Fourier decomposition of \(\psi\) in \(\theta\). Following \cite{delpinofelmerkowalczyk2004}, we decompose \(\psi\) as
\begin{equation}
\label{psifourier}
    \psi=\psi^{0}+\sum_{k=1}^{\infty}\psi_{k}^{1}+\sum_{k=1}^{\infty}\psi_{k}^{2},
\end{equation}where
\begin{align}
    \begin{split}
    \psi^{0}&=\psi_{1}^{0}(r)+i\psi_{2}^{0}(r),\\
    \psi_{k}^{1}&=\psi_{k1}^{1}(r)\cos{k\theta}+i\psi_{k2}^{1}(r)\sin{k\theta},\\
    \psi_{k}^{2}&=\psi_{k1}^{2}(r)\sin{k\theta}+i\psi_{k2}^{2}(r)\cos{k\theta}.
    \end{split}
\end{align}In a similar spirit, we decompose \(h\) as 
\begin{equation}
\label{hfourier}
    h=iW\left(h^{0}+\sum_{k=1}^{\infty}h_{k}^{1}+\sum_{k=1}^{\infty}h_{k}^{2}\right),
\end{equation}where
\begin{align}
    \begin{split}
    h^{0}&=h_{1}^{0}(r)+ih_{2}^{0}(r),\\
    h_{k}^{1}&=h_{k1}^{1}(r)\cos{k\theta}+ih_{k2}^{1}(r)\sin{k\theta},\\
    h_{k}^{2}&=h_{k1}^{2}(r)\sin{k\theta}+ih_{k2}^{2}(r)\cos{k\theta}.
    \end{split}
\end{align}By direct calculation, we see that equation (\ref{linearizedeqpsi}) reduces to the ODEs
\begin{equation}
\label{fmodesystem}
\varphi''+\left(2\frac{w'}{w}+\frac{1}{r}\right)\varphi'-\frac{1}{r^{2}}\begin{pmatrix}
        k^{2} & (-1)^{l+1}2k \\
        (-1)^{l+1}2k & k^{2}+2w^{2}r^{2} 
    \end{pmatrix}\varphi=\widetilde{h}
\end{equation}for \(k\geq1\) and \(l\in\{1,2\}\), where 
\begin{equation}
\label{fmodesystemvector}
    \varphi=\begin{pmatrix}[1.3]
        \psi_{k1}^{l}\\
        \psi_{k2}^{l}
    \end{pmatrix},\quad\quad\widetilde{h}= 
    \begin{pmatrix}[1.3]
        h_{k1}^{l}\\
        h_{k2}^{l}
    \end{pmatrix}.
\end{equation}For Fourier mode \(k=0\), we get the uncoupled equations
\begin{gather}
    \psi_{1}^{0\:''}+\left(2\frac{w'}{w}+\frac{1}{r}\right)\psi_{1}^{0\:'}=h_{1}^{0},\label{fmodeuncoupled1}\\
   \psi_{2}^{0\:''}+\left(2\frac{w'}{w}+\frac{1}{r}\right)\psi_{2}^{0\:'}-2w^{2}\psi_{2}^{0}=h_{2}^{0}.\label{fmodeuncoupled2} 
\end{gather}

A central result of our analysis concerns an explicit representation formula for solutions of (\ref{fmodesystem}), see Proposition \ref{mode1repformula} and Proposition \ref{modegeq2repformula} below, which allows us to handle (\ref{linearizedeq}) with minimal assumptions on the right-hand side. This representation formula is interesting in its own right, and
requires a non-trivial derivation due to the coupled nature of the ODE system.

We emphasise that the statement of Theorem \ref{mainthm} can be extended to functions \(h\) satisfying any polynomial growth or decay rate as \(r\to\infty\), after making suitable modifications to the norms (\ref{normphi}), (\ref{normh}). Having a flexible solvability theory of this kind is crucial in various Ginzburg-Landau type problems involving the construction of point vortices or filaments via linearization. See, for example, the works \cites{delpinokowalczykmusso2006,chironpacherie2021,liuwei2020,daviladelpinomedinarodiac2022,pacardriviere2000} where this approach is broadly developed. In another direction, properties of the operator \(L\) are known to play an important role when addressing the dynamic stability of the degree-one vortex in \(\mathbb{R}^{2}\). We mention the recent work of Palacios-Pusateri \cite{palaciospusateri2024}, which establishes decay estimates for the linearized evolution in the relativistic case, and the work of Gravejat-Pacherie-Smets \cite{gravejatpacheriesmets2022}, which establishes the orbital stability of the degree-one vortex under the Gross-Pitaevskii flow (\ref{gpeq}). A related issue concerns the analysis of magnetic vortices: we refer the reader to \cites{gustafson2002,gustafsonsigal2000} for spectral and orbital stability results in this setting. 

Finally, we note that previous results in the spirit of Theorem \ref{mainthm} have been obtained by Pacard and Rivi\`{e}re \cite{pacardriviere2000}, who studied equation (\ref{linearizedeq}) in the context of weighted H\"{o}lder spaces. Their approach relies partly on an investigation of the ODEs (\ref{fmodesystem}) (in a slightly different form) but no explicit representation formula for the Fourier modes is given. This restricts the type of decay rates which can be treated. Moreover, the analysis in \cite{pacardriviere2000} is mostly confined to the Dirichlet problem for (\ref{linearizedeq}) on bounded domains, which avoids the integrability issue for the orthogonality conditions in \(\mathbb{R}^{2}\).

The rest of the paper is devoted to the proof of Theorem \ref{mainthm}.

\section{Fourier mode \texorpdfstring{\(k=0\)}{k=0}}

The aim of this section is to prove Theorem \ref{mainthm} in the Fourier mode \(k=0\) case. We have the following result.

\begin{proposition}
\label{fouriermode0prop}
There exists a constant \(C>0\) such that for any 
\begin{equation}
\label{hmode0}
    h=iWh^{0},\quad\quad h^{0}=h_{1}^{0}(r)+ih_{2}^{0}(r),
\end{equation}
satisfying \(\norm{h}_{**}<\infty\), equation (\ref{linearizedeq}) has a solution 
\begin{equation}
\label{phimode0}
    \phi=iW\psi^{0},\quad\quad \psi^{0}=\psi_{1}^{0}(r)+i\psi_{2}^{0}(r),
\end{equation}such that 
\begin{equation}
\label{fouriermode0est}
    \norm{\phi}_{*}\leq C\norm{h}_{**}.
\end{equation}
\end{proposition}

Recall that, for \(h\) of the form (\ref{hmode0}) and \(\phi\) of the form (\ref{phimode0}), problem (\ref{linearizedeq}) reduces to the uncoupled ODEs 
\begin{gather}
    \psi_{1}^{0\:''}+\left(2\frac{w'}{w}+\frac{1}{r}\right)\psi_{1}^{0\:'}=h_{1}^{0},\label{fmodeuncoupled1v2}\\
   \psi_{2}^{0\:''}+\left(2\frac{w'}{w}+\frac{1}{r}\right)\psi_{2}^{0\:'}-2w^{2}\psi_{2}^{0}=h_{2}^{0}.\label{fmodeuncoupled2v2} 
\end{gather}The first equation (\ref{fmodeuncoupled1v2}) can be solved explicitly via the formula
\begin{equation*}
        \psi_{1}^{0}(r)=\int_{0}^{r}\frac{ds}{w(s)^{2}s}\int_{0}^{s}w(t)^{2}th_{1}^{0}(t)\:dt,
    \end{equation*}from which it is directly verified, using (\ref{wproperties}), that estimate (\ref{fouriermode0est}) holds for \(\psi_{1}^{0}(r)\).

  Concerning the second equation (\ref{fmodeuncoupled2v2}), we have the following lemma. 

    \begin{lemma}
    \label{fmodeuncoupled2v2homlemma}
    The homogeneous problem
    \begin{equation}
    \label{fmodeuncoupled2v2hom}
    \psi_{2}^{0\:''}+\left(2\frac{w'}{w}+\frac{1}{r}\right)\psi_{2}^{0\:'}-2w^{2}\psi_{2}^{0}=0,    
    \end{equation}admits two linearly independent solutions \(z_{1,0}\) and \(z_{2,0}\) with the asymptotic behaviour
  \begin{equation}
  \label{mode0kernelprop}
    z_{1,0}(r)=\begin{cases}
        O(1)\text{ as }r\to0^{+},\\
        O\left(r^{-1/2}e^{\sqrt{2}r}\right)\text{ as }r\to\infty,
    \end{cases}\quad
        z_{2,0}(r)=\begin{cases}
            O(r^{-2})\text{ as }r\to0^{+},\\
            O\left(r^{-1/2}e^{-\sqrt{2}r}\right)\text{ as }r\to\infty.
        \end{cases}
\end{equation}
    \end{lemma}   
  
\begin{proof}
As \(r\to\infty\) problem (\ref{fmodeuncoupled2v2hom}) becomes asymptotic to \(\psi_{2}^{0\:''}+\frac{1}{r}\psi_{2}^{0\:'}-2\psi_{2}^{0}=0\). This is a Bessel-type equation, which has a solution \(\psi_{2}^{0}=O\big(r^{-1/2}e^{-\sqrt{2}r}\big)\) as \(r\to\infty\). A suitable use of barriers gives a solution \(z_{2,0}(r)\) of (\ref{fmodeuncoupled2v2hom}) with this behaviour. By the nondegeneracy results in \cite{delpinofelmerkowalczyk2004}*{Theorem 1.1} and \cite{pacardriviere2000}*{Theorem 3.2}, the function \(z_{2,0}(r)\) must blow up faster than \(O(r^{-1})\) as \(r\to0^{+}\) (otherwise, one could construct a bounded mode 0 element of the kernel of \(L\) which is different from \(iW\)). Since equation (\ref{fmodeuncoupled2v2hom}) becomes similar to \(\psi_{2}^{0\:''}+\frac{3}{r}\psi_{2}^{0\:'}=0\) as \(r\to0^{+}\), we necessarily have \(z_{2,0}(r)=O(r^{-2})\) as \(r\to0^{+}\). The maximum principle then implies that \(z_{2,0}(r)\) does not vanish for any \(r\in(0,\infty)\), thus Liouville's formula for the Wronskian gives a second linearly independent solution of (\ref{fmodeuncoupled2v2hom}) defined via the formula
    \begin{equation*}
        z_{1,0}(r)=z_{2,0}(r)\int_{0}^{r}\frac{1}{w(s)^{2}z_{2,0}(s)^{2}s}\:ds.
    \end{equation*}From this expression it is straightforward to check that \(z_{1,0}(r)\) has the asymptotic properties stated.
\end{proof}  

\begin{proof}[Proof of Proposition \ref{fouriermode0prop}]
    It remains to solve equation (\ref{fmodeuncoupled2v2}). Using the functions \(z_{1,0}\) and \(z_{2,0}\), a solution of this problem can be obtained via the variation of parameters formula
    \begin{equation*}
    \psi_{2}^{0}(r)=z_{1,0}(r)\int_{r}^{\infty}w(s)^{2}sh_{2}^{0}(s)z_{2,0}(s)\:ds+z_{2,0}(r)\int_{0}^{r}w(s)^{2}sh_{2}^{0}(s)z_{1,0}(s)\:ds.
\end{equation*}It is then directly verified, using (\ref{wproperties}) and (\ref{mode0kernelprop}), that estimate (\ref{fouriermode0est}) holds for \(\psi_{2}^{0}\).
\end{proof}

\section{Fourier mode \texorpdfstring{\(k=1\)}{k=1}}

In this section we prove Theorem \ref{mainthm} in the Fourier mode \(k=1\) case. We have the following results.

\begin{proposition}
\label{fouriermode1prop}
There exists a constant \(C>0\) such that for any 
\begin{equation}
\label{hmode1v1}
    h=iWh_{1}^{1},\quad\quad h_{1}^{1}=h_{11}^{1}(r)\cos\theta+ih_{12}^{1}(r)\sin\theta,
\end{equation}
satisfying \(\norm{h}_{**}<\infty\) and the orthogonality condition \(\left\langle h,\frac{\partial W}{\partial x_{2}} \right\rangle=0\), equation (\ref{linearizedeq}) has a solution 
\begin{equation}
\label{phimode1v1}
    \phi=iW\psi_{1}^{1},\quad\quad \psi_{1}^{1}=\psi_{11}^{1}(r)\cos\theta+i\psi_{12}^{1}(r)\sin\theta,
\end{equation}such that 
\begin{equation}
\label{fouriermode1estv1}
    \norm{\phi}_{*}\leq C\norm{h}_{**}.
\end{equation}
\end{proposition}

\begin{corollary}
\label{fouriermode1cor}
    There exists a constant \(C>0\) such that for any
    \begin{equation}
\label{hmode1v2}
    h=iWh_{1}^{2},\quad\quad h_{1}^{2}=h_{11}^{2}(r)\sin\theta+ih_{12}^{2}(r)\cos\theta,
\end{equation}
satisfying \(\norm{h}_{**}<\infty\) and the orthogonality condition \(\left\langle h,\frac{\partial W}{\partial x_{1}} \right\rangle=0\), equation (\ref{linearizedeq}) has a solution 
\begin{equation}
\label{phimode1v2}
    \phi=iW\psi_{1}^{2},\quad\quad \psi_{1}^{2}=\psi_{11}^{2}(r)\sin\theta+i\psi_{12}^{2}(r)\cos\theta,
\end{equation}such that 
\begin{equation}
\label{fouriermode1estv2}
    \norm{\phi}_{*}\leq C\norm{h}_{**}.
\end{equation}
\end{corollary}

The proof of Proposition \ref{fouriermode1prop} proceeds in several steps: we then deduce Corollary \ref{fouriermode1cor} using a simple symmetry argument. Recall that, for \(h\) of the form (\ref{hmode1v1}) and \(\phi\) of the form (\ref{phimode1v1}), problem (\ref{linearizedeq}) reduces to the coupled ODE system 
\begin{equation}
\label{fmode1system}
\varphi''+\left(2\frac{w'}{w}+\frac{1}{r}\right)\varphi'-\frac{1}{r^{2}}\begin{pmatrix}
        1 & 2 \\
        2 & 1+2w^{2}r^{2} 
    \end{pmatrix}\varphi=\widetilde{h},
\end{equation}where 
\begin{equation}
    \varphi=\begin{pmatrix}[1.3]
        \psi_{11}^{1}\\
        \psi_{12}^{1}
    \end{pmatrix},\quad\quad\widetilde{h}= 
    \begin{pmatrix}[1.3]
        h_{11}^{1}\\
        h_{12}^{1}
    \end{pmatrix}.
\end{equation}
We intend to solve this problem using the variation of parameters method. To this end, our first objective is to classify the solutions to the homogeneous version of (\ref{fmode1system}). This is carried out in Subsection \ref{solhomproblemfmode1} below, by studying the possible behaviours of the kernel as \(r\to0^{+}\) and as \(r\to\infty\). We then solve the inhomogeneous problem (\ref{fmode1system}) in Subsection \ref{solinhomproblemfmode1}.

\subsection{Solutions of the homogeneous problem: Fourier mode \texorpdfstring{\(k=1\)}{k=1}}
\label{solhomproblemfmode1}

The homogeneous problem associated to (\ref{fmode1system}) reads 
\begin{equation}
\label{fmode1systemhom}
\varphi''+\left(2\frac{w'}{w}+\frac{1}{r}\right)\varphi'-\frac{1}{r^{2}}\begin{pmatrix}
        1 & 2 \\
        2 & 1+2w^{2}r^{2} 
    \end{pmatrix}\varphi=0,
\end{equation}where \(\varphi=(\varphi_{1},\varphi_{2})\). We recall the nondegeneracy result for \(L\) in this setting.

\begin{lemma}
\label{fmode1nondegeneracy}
    Equation (\ref{fmode1systemhom}) admits the explicit solution \(\varphi=(1/r,-w'/w)\).
    Moreover, any solution of (\ref{fmode1systemhom}) which blows up at most like \(O(r^{-1})\) as \(r\to0^{+}\) and remains bounded as \(r\to\infty\) must necessarily be a multiple of this solution.
\end{lemma}

\begin{proof}
By the nondegeneracy result of Pacard and Rivi\`{e}re \cite{pacardriviere2000}*{Theorem 3.2}, any bounded solution of \(L[\phi]=0\) of the form (\ref{phimode1v1}) must necessarily be a multiple of
\begin{equation*}
    \frac{\partial W}{\partial x_{2}}=iW\bigg(\frac{1}{r}\cos\theta-i\frac{w'}{w}\sin\theta\bigg).
\end{equation*}Translating this fact to (\ref{fmode1systemhom}), and using the properties \(\abs{W}=w\sim r\) as \(r\to0^{+}\) and \(\abs{W}=w\sim1\) as \(r\to\infty\), we deduce the statement of the lemma.
\end{proof}

The primary aim of this subsection is to prove the following result.

\begin{proposition}
\label{mode1kernel}
The homogeneous problem (\ref{fmode1systemhom}) 
admits four linearly independent solutions \(z_{1,1}\), \(z_{2,1}\), \(z_{3,1}\), \(z_{4,1}\) with the asymptotic behaviour
\begin{equation}
\label{mode1kernelasymprop}
\begin{aligned}[c]
z_{1,1}(r)&=\begin{pmatrix}[1.3]
    O(r^{-1})\\
    O(r^{-1})
    \end{pmatrix},\text{ as }r\to0^{+},\\
z_{2,1}(r)&=\begin{pmatrix}[1.3]
    O(r^{-1}\log r)\\
    O(r^{-1}\log r)
    \end{pmatrix},\text{ as }r\to0^{+},\\
z_{3,1}(r)&=\begin{pmatrix}[1.3]
    O(r)\\
    O(r)
    \end{pmatrix},\text{ as }r\to0^{+},\\
z_{4,1}(r)&=\begin{pmatrix}[1.3]
    O(r^{-3})\\
    O(r^{-3})
    \end{pmatrix},\text{ as }r\to0^{+},    
\end{aligned}
\quad\quad
\begin{aligned}[c]
z_{1,1}(r)&=\begin{pmatrix}[1.3]
    O(r^{-1})\\
    O(r^{-3})
    \end{pmatrix},\text{ as }r\to\infty,\\
z_{2,1}(r)&=\begin{pmatrix}[1.3]
    O(r)\\
    O(r^{-1})
    \end{pmatrix},\text{ as }r\to\infty,\\
z_{3,1}(r)&=\begin{pmatrix}[1.3]
    O\big(r^{-5/2}e^{\sqrt{2}r}\big)\\
    O\big(r^{-1/2}e^{\sqrt{2}r}\big)
    \end{pmatrix},\text{ as }r\to\infty,\\
z_{4,1}(r)&=\begin{pmatrix}[1.3]
    O\big(r^{-5/2}e^{-\sqrt{2}r}\big)\\
    O\big(r^{-1/2}e^{-\sqrt{2}r}\big)
    \end{pmatrix},\text{ as }r\to\infty.    
\end{aligned}
\end{equation}The first solution \(z_{1,1}(r)\) is given explicitly by
\begin{equation}
\label{mode1expker}
        z_{1,1}(r)=\begin{pmatrix}[1.3]
    1/r\\
    -w'/w
    \end{pmatrix}.
    \end{equation}
\end{proposition}

The proof of Proposition \ref{mode1kernel} proceeds by first classifying the possible behaviour of solutions as \(r\to0^{+}\) and as \(r\to\infty\). Concerning the behaviour as \(r\to0^{+}\), we have the following.

\begin{lemma}
\label{mode1kernelat0}
   The solution space of equation (\ref{fmode1systemhom}) is spanned, for small \(r>0\), by four functions with the asymptotic behaviour
    \begin{equation}
    \label{mode1kernelat0eq}
    \begin{aligned}[c]
    \begin{pmatrix}[1.3]
        r^{-1}+O(r)\\
        -r^{-1}+O(r)
    \end{pmatrix},\text{ as }r\to0^{+},\\
    \begin{pmatrix}[1.3]
        r+O(r^{3})\\
        r+O(r^{3})
    \end{pmatrix},\text{ as }r\to0^{+},
    \end{aligned}\quad\quad
    \begin{aligned}[c]
    \begin{pmatrix}[1.3]
        r^{-1}\log r+O(r\log r)\\
        -r^{-1}\log r+O(r\log r)
    \end{pmatrix},\text{ as }r\to0^{+},\\
    \begin{pmatrix}[1.3]
        r^{-3}+O(r^{-1})\\
        r^{-3}+O(r^{-1})
    \end{pmatrix},\text{ as }r\to0^{+}.
    \end{aligned}
\end{equation}
\end{lemma}

\begin{proof}
As \(r\to0^{+}\), problem (\ref{fmode1systemhom}) becomes asymptotic to the equation
\begin{equation}
\label{fmode1systemhomrto0}
\varphi''+\frac{3}{r}\varphi'-\frac{1}{r^{2}}\begin{pmatrix}
        1 & 2 \\
        2 & 1 
    \end{pmatrix}\varphi=0.
\end{equation}The four linearly independent solutions of (\ref{fmode1systemhomrto0}) are given by
\begin{equation*}
\begin{pmatrix}[1.3]
        r^{-1}\\
        -r^{-1}
    \end{pmatrix},\quad 
\begin{pmatrix}[1.3]
        r^{-1}\log r\\
        -r^{-1}\log r
    \end{pmatrix},\quad
\begin{pmatrix}[1.3]
        r\\
        r
    \end{pmatrix},\quad 
\begin{pmatrix}[1.3]
        r^{-3}\\
        r^{-3}
    \end{pmatrix},     
\end{equation*}as can be verified by direct calculation. A contraction mapping argument similar to that employed in \cite{pacardriviere2000}*{Section 3.3}, then gives four linearly independent solutions of (\ref{fmode1systemhom}) for small \(r>0\) which have the properties (\ref{mode1kernelat0eq}).
\end{proof}

The classification of solutions as \(r\to\infty\) reads as follows.

\begin{lemma}
\label{mode1kernelatinf}
    The solution space of equation (\ref{fmode1systemhom}) is spanned, for large \(r>0\), by four functions with the asymptotic behaviour
    \begin{equation}
    \begin{aligned}[c]
    \begin{pmatrix}[1.3]
        r^{-1}+O(r^{-3})\\
        -r^{-3}+O(r^{-5})
    \end{pmatrix},\text{ as }r\to\infty,\\
    \begin{pmatrix}[1.3]
        r+O(r^{-1})\\
        -r^{-1}+O(r^{-3})
    \end{pmatrix},\text{ as }r\to\infty,\\
    \begin{pmatrix}[1.3]
        r^{-5/2}e^{\sqrt{2}r}+O\big(r^{-7/2}e^{\sqrt{2}r}\big)\\
        r^{-1/2}e^{\sqrt{2}r}+O\big(r^{-3/2}e^{\sqrt{2}r}\big)
    \end{pmatrix},\text{ as }r\to\infty,\\
    \begin{pmatrix}[1.3]
        r^{-5/2}e^{-\sqrt{2}r}+O\big(r^{-7/2}e^{-\sqrt{2}r}\big)\\
        r^{-1/2}e^{-\sqrt{2}r}+O\big(r^{-3/2}e^{-\sqrt{2}r}\big)
    \end{pmatrix},\text{ as }r\to\infty.
    \end{aligned}
    \end{equation}

\begin{proof}
    As \(r\to\infty\), problem (\ref{fmode1systemhom}) becomes asymptotic to the coupled system of equations
    \begin{align}
        &\varphi_{1}''+\frac{1}{r}\varphi_{1}'-\frac{1}{r^{2}}\varphi_{1}-\frac{2}{r^{2}}\varphi_{2}=0,\label{proofrtoinfmode1hom1}\\
        &\varphi_{2}''+\frac{1}{r}\varphi_{2}'-\frac{2}{r^{2}}\varphi_{1}-2\varphi_{2}=0.\label{proofrtoinfmode1hom2}
    \end{align}Neglecting the \(\varphi_{2}\) term in (\ref{proofrtoinfmode1hom1}), we arrive at the equation \(\varphi_{1}''+\tfrac{1}{r}\varphi_{1}'-\tfrac{1}{r^{2}}\varphi_{1}=0\), whose solutions are spanned by the functions \(\varphi_{1}=r^{-1}\) and \(\varphi_{1}=r\). Substituting this ansatz for \(\varphi_{1}\) into (\ref{proofrtoinfmode1hom1})-(\ref{proofrtoinfmode1hom2}) and matching powers of \(r\) for \(\varphi_{2}\), we can construct two formal power series solutions to this system with asymptotic behaviour
    \begin{equation*}
        \varphi= \begin{pmatrix}[1.3]
        r^{-1}+O(r^{-3})\\
        -r^{-3}+O(r^{-5})
    \end{pmatrix},\quad\quad\varphi=\begin{pmatrix}[1.3]
        r+O(r^{-1})\\
        -r^{-1}+O(r^{-3})
    \end{pmatrix},
    \end{equation*}as \(r\to\infty\). Using a contraction mapping argument similar to that employed in \cite{pacardriviere2000}*{Section 3.3}, we can then find two linearly independent solutions of the original equation (\ref{fmode1systemhom}) for large \(r>0\) with these properties.

    Turning our attention to (\ref{proofrtoinfmode1hom2}) and neglecting the \(\varphi_{1}\) term, we arrive at the equation \(\varphi_{2}''+\tfrac{1}{r}\varphi_{2}'-2\varphi_{2}=0\). This is a Bessel-type problem whose solutions are spanned by two functions with the behaviour \(\varphi_{2}=O\big(r^{-1/2}e^{\sqrt{2}r}\big)\) and \(\varphi_{2}=O\big(r^{-1/2}e^{-\sqrt{2}r}\big)\) as \(r\to\infty\). Substituting this ansatz for \(\varphi_{2}\) into (\ref{proofrtoinfmode1hom1})-(\ref{proofrtoinfmode1hom2}) and matching powers of \(r\) for \(\varphi_{1}\), we find two generalized formal power series solutions to this system with asymptotic behaviour 
    \begin{equation*}
        \varphi=\begin{pmatrix}[1.3]
        r^{-5/2}e^{\sqrt{2}r}+O\big(r^{-7/2}e^{\sqrt{2}r}\big)\\
        r^{-1/2}e^{\sqrt{2}r}+O\big(r^{-3/2}e^{\sqrt{2}r}\big)
    \end{pmatrix},\quad\quad\varphi=\begin{pmatrix}[1.3]
        r^{-5/2}e^{-\sqrt{2}r}+O\big(r^{-7/2}e^{-\sqrt{2}r}\big)\\
        r^{-1/2}e^{-\sqrt{2}r}+O\big(r^{-3/2}e^{-\sqrt{2}r}\big)
    \end{pmatrix},
    \end{equation*}as \(r\to\infty\). Again a contraction mapping argument can be employed to obtain two linearly independent solutions of (\ref{fmode1systemhom}) for large \(r>0\) with these properties.
\end{proof}
\end{lemma}

The next step is to match the possible behaviour of solutions as \(r\to0^{+}\) with the possible behaviour of solutions as \(r\to\infty\). The following maximum principle type result is crucial.

\begin{lemma}
\label{nosignchangemode1}
    Let \(\varphi=(\varphi_{1},\varphi_{2})\) be a solution of (\ref{fmode1systemhom}).
    \begin{enumerate}[(i)]
        \item If \(\varphi_{1},\varphi_{2}\to0\) as \(r\to0^{+}\) and \(\varphi_{1}(r),\varphi_{2}(r)>0\) for small \(r>0\), then \(\varphi_{1}(r),\varphi_{2}(r)>0\) for all \(r\in(0,\infty)\).
        \item If \(\varphi_{1},\varphi_{2}\to0\) as \(r\to\infty\) and \(\varphi_{1}(r),\varphi_{2}(r)>0\) for large \(r>0\), then \(\varphi_{1}(r),\varphi_{2}(r)>0\) for all \(r\in(0,\infty)\).
    \end{enumerate}
\end{lemma}

\begin{proof}
    Consider a solution \(\varphi=(\varphi_{1},\varphi_{2})\) of (\ref{fmode1systemhom}) satisfying the assumptions of \((i)\). If \(\varphi_{1}\) or \(\varphi_{2}\) vanish at some interior point in \((0,\infty)\), then we can find a strictly positive local maximum for \(\varphi_{1}\) or \(\varphi_{2}\) such that the other component is also strictly positive at this point. In the first case this implies the existence of an interior point where
    \begin{equation*}
        \varphi_{1}^{''}\leq0,\quad\varphi_{1}^{'}=0,\quad\varphi_{1}>0,\quad\varphi_{2}>0.
    \end{equation*}However, the first equation of (\ref{fmode1systemhom}) reads
    \begin{equation*}
        \varphi_{1}''+\left(2\frac{w'}{w}+\frac{1}{r}\right)\varphi_{1}'-\frac{1}{r^{2}}\varphi_{1}=\frac{2}{r^{2}}\varphi_{2}>0,
    \end{equation*}at this point, which is a contradiction. In the second case we find an interior point where
    \begin{equation*}
        \varphi_{2}^{''}\leq0,\quad\varphi_{2}^{'}=0,\quad\varphi_{2}>0,\quad\varphi_{1}>0.
    \end{equation*}However, the second equation of (\ref{fmode1systemhom}) reads
    \begin{equation*}
    \varphi_{2}''+\left(2\frac{w'}{w}+\frac{1}{r}\right)\varphi_{2}'-\frac{1}{r^{2}}\varphi_{2}-2w^{2}\varphi_{2}=\frac{2}{r^{2}}\varphi_{1}>0,
    \end{equation*}at this point, which is again a contradiction. This completes the proof of \((i)\), and the proof of \((ii)\) is similar.
\end{proof}

We can now complete the classification of solutions for the homogeneous problem.

\begin{proof}[Proof of Proposition \ref{mode1kernel}]
    One solution \(z_{1,1}(r)=(1/r,-w'/w)\) of equation (\ref{fmode1systemhom}) is already known, as follows from Lemma \ref{fmode1nondegeneracy}. Using the properties (\ref{wproperties}) of \(w\), we see that \(z_{1,1}(r)\) has the asymptotic behaviour stated.
    
    Next, let \(z_{3,1}(r)\) be a solution of the homogeneous problem from Lemma \ref{mode1kernelat0} with the behaviour
    \begin{equation*}
        z_{3,1}(r)=\begin{pmatrix}[1.3]
    r+O(r^{3})\\
    r+O(r^{3})
    \end{pmatrix},
    \end{equation*}for small \(r>0\). By standard ODE theory \(z_{3,1}(r)\) extends to a solution of (\ref{fmode1systemhom}) defined for all \(r\in(0,\infty)\), and Lemma \ref{fmode1nondegeneracy} implies that \(z_{3,1}(r)\) must become unbounded as \(r\to\infty\). By Lemma \ref{nosignchangemode1}, the solution \(z_{3,1}(r)\) must have components of the same sign as \(r\to\infty\). It follows from Lemma \ref{mode1kernelatinf} that 
    \begin{equation*}
        z_{3,1}(r)=c\begin{pmatrix}[1.3]
        r^{-5/2}e^{\sqrt{2}r}+O\big(r^{-7/2}e^{\sqrt{2}r}\big)\\
        r^{-1/2}e^{\sqrt{2}r}+O\big(r^{-3/2}e^{\sqrt{2}r}\big)
    \end{pmatrix},\text{ as }r\to\infty,
    \end{equation*}for some constant \(c>0\). 

    Now let us consider a solution \(z_{4,1}(r)\) of the homogeneous problem from Lemma \ref{mode1kernelatinf} with the behaviour
    \begin{equation*}
        z_{4,1}(r)=\begin{pmatrix}[1.3]
        r^{-5/2}e^{-\sqrt{2}r}+O\big(r^{-7/2}e^{-\sqrt{2}r}\big)\\
        r^{-1/2}e^{-\sqrt{2}r}+O\big(r^{-3/2}e^{-\sqrt{2}r}\big)
    \end{pmatrix},\text{ as }r\to\infty.
    \end{equation*}The function \(z_{4,1}(r)\) extends to a solution of (\ref{fmode1systemhom}) defined on \((0,\infty)\), and by Lemma \ref{fmode1nondegeneracy} this solution must blow up faster than \(O(r^{-1})\) as \(r\to0^{+}\). By Lemma \ref{nosignchangemode1}, the components of \(z_{4,1}(r)\) have the same sign as \(r\to0^{+}\). It follows from Lemma \ref{mode1kernelat0} that \(z_{4,1}(r)\) has the behaviour
    \begin{equation*}
        z_{4,1}(r)=c\begin{pmatrix}[1.3]
       r^{-3}+O(r^{-1})\\
        r^{-3}+O(r^{-1})
    \end{pmatrix},\text{ as }r\to0^{+},
    \end{equation*}for some constant \(c>0\).

    Finally, let \(z_{2,1}(r)\) be a solution of the homogeneous problem from Lemma \ref{mode1kernelat0} with the behaviour
    \begin{equation*}
    z_{2,1}(r)=\begin{pmatrix}[1.3]
        r^{-1}\log r+O(r\log r)\\
        -r^{-1}\log r+O(r\log r)
    \end{pmatrix},
    \end{equation*}for small \(r>0\). The function \(z_{2,1}(r)\) must have the remaining behaviour at infinity described in Lemma \ref{mode1kernelatinf}, namely
    \begin{equation*}
        z_{2,1}(r)=c\begin{pmatrix}[1.3]
       r+O(r^{-1})\\
        -r^{-1}+O(r^{-3})
    \end{pmatrix},\text{ as }r\to\infty.
    \end{equation*}
\end{proof}

\subsection{Solutions of the inhomogeneous problem: Fourier mode \texorpdfstring{\(k=1\)}{k=1}} 
\label{solinhomproblemfmode1}
In this subsection we obtain a representation formula for solutions of the inhomogeneous problem (\ref{fmode1system}), and complete the proof of Proposition \ref{fouriermode1prop} and Corollary \ref{fouriermode1cor}. We use the notation \(a\cdot b=a_{1}b_{1}+a_{2}b_{2}\) for the usual dot product in \(\mathbb{R}^{2}\).

\begin{proposition}
 \label{mode1repformula}  
    Let \(\varphi=(\varphi_{1},\varphi_{2})\) and \(\widetilde{h}=(\widetilde{h}_{1},\widetilde{h}_{2})\). The general solution of the inhomogeneous problem (\ref{fmode1system}) can be written in the form
    \begin{equation*}
\begin{split}
    \varphi(r)=&\quad\left(\int w(s)^{2}s\:\widetilde{h}(s)\cdot z_{2,1}(s)\:ds\right)z_{1,1}(r)\\
    &-\left(\int w(s)^{2}s\:\widetilde{h}(s)\cdot z_{1,1}(s)\:ds\right)z_{2,1}(r)\\
    &+\left(\int w(s)^{2}s\:\widetilde{h}(s)\cdot z_{4,1}(s)\:ds\right)z_{3,1}(r)\\
    &-\left(\int w(s)^{2}s\:\widetilde{h}(s)\cdot z_{3,1}(s)\:ds\right)z_{4,1}(r),
\end{split}
\end{equation*}where \(z_{1,1},\ldots, z_{4,1}\) denote the four linearly independent solutions of the homogeneous problem described in Proposition \ref{mode1kernel}, and the symbols \(\int\) denote arbitrary antiderivatives.
\end{proposition}

\begin{proof}
Let \(z_{j,1}=(z_{j1,1},z_{j2,1})\) denote the two components of \(z_{j,1}\). The variation of parameters formula gives the general solution of (\ref{fmode1system}) in the form 
\begin{equation*}
    \varphi(r)=\sum_{j=1}^{4}c_{j}(r)z_{j,1}(r),
\end{equation*}where \(c_{j}(r)\) denote real-valued functions satisfying the equations
    \begin{equation*}
        \sum_{j=1}^{4}c_{j}'(r)z_{j,1}(r)=0,\quad\quad\sum_{j=1}^{4}c_{j}'(r)z_{j,1}'(r)=\widetilde{h}.
    \end{equation*}
By Cramer's rule, the unique solution of this system is given by
    \begin{equation*}
        c_{j}'(r)=\frac{\det A_{j}}{\det A},\quad\quad j=1,\ldots,4,
    \end{equation*}where \(A\) denotes the \(4\times4\) matrix whose \(j\)th column is given by the vector
    \begin{equation*}
        \big(z_{j1,1},z_{j2,1},z_{j1,1}',z_{j2,1}'\big),
    \end{equation*}
    and \(A_{j}\) denotes the \(4\times4\) matrix formed by replacing the \(j\)th column of \(A\) with the vector 
    \begin{equation*}
        \big(0,0,\widetilde{h}_{1},\widetilde{h}_{2}\big).
    \end{equation*}

     Writing problem (\ref{fmode1systemhom}) as a first order system in four variables and using Liouville's formula for the Wronskian, we have 
        \(\det A=\kappa/ (w(r)^{4}r^{2})\) for some constant \(\kappa\neq0\).

    We now intend to find a convenient expression for \(\det A_{j}\). Focusing first on \(\det A_{1}\), we can write this determinant in the form \(\det A_{1}=(a\cdot\widetilde{h})/(w(r)^{2}r)\) where \(a:=(a_{1},a_{2})\),
    \begin{equation*}
        a_{1}(r):=w(r)^{2}r\begin{vmatrix}
    z_{21,1} & z_{31,1} & z_{41,1} \\
    z_{22,1} & z_{32,1} & z_{42,1} \\
    z_{22,1}' & z_{32,1}' & z_{42,1}'
    \end{vmatrix},\quad\quad a_{2}(r)=-w(r)^{2}r\begin{vmatrix}
    z_{21,1} & z_{31,1} & z_{41,1} \\
    z_{22,1} & z_{32,1} & z_{42,1} \\
    z_{21,1}' & z_{31,1}' & z_{41,1}'
    \end{vmatrix}.
    \end{equation*}A lengthy but straightforward calculation shows that \(a=(a_{1},a_{2})\) is a solution of the homogeneous problem (\ref{fmode1systemhom}), thus this function must be a linear combination of \(z_{1,1},\ldots,z_{4,1}\). Moreover, using the properties (\ref{mode1kernelasymprop}) of \(z_{1,1},\ldots,z_{4,1}\) described in Proposition \ref{mode1kernel}, and their natural extension to the first derivatives, we find 
    \begin{equation*}
     a=\begin{pmatrix}[1.3]
        O(r^{-1}\log r)\\
        O(r^{-1}\log r)
    \end{pmatrix},\text{ as }r\to0^{+}, \quad\quad a=\begin{pmatrix}[1.3]
        O(r)\\
        O(r^{-1})
    \end{pmatrix},\text{ as }r\to\infty.
    \end{equation*}It follows that \(a=(a_{1},a_{2})\) must be a multiple of \(z_{2,1}\). Hence,
    \begin{equation*}
        c_{1}'(r)=\frac{\det A_{1}}{\det A}=\kappa_{1}w(r)^{2}r\:\widetilde{h}(r)\cdot z_{2,1}(r)
    \end{equation*}for some constant \(\kappa_{1}\neq0\). Assuming that \(z_{1,1},\ldots,z_{4,1}\) are initially normalized to have the exact behaviour at infinity described in Lemma \ref{mode1kernelatinf}, we can carry out the same analysis for the remaining determinants to find 
    \begin{align*}
        c_{2}'(r)&=-\kappa_{1}w(r)^{2}r\:\widetilde{h}(r)\cdot z_{1,1}(r)\\
        c_{3}'(r)&=\kappa_{2}w(r)^{2}r\:\widetilde{h}(r)\cdot z_{4,1}(r)\\
        c_{4}'(r)&=-\kappa_{2}w(r)^{2}r\:\widetilde{h}(r)\cdot z_{3,1}(r),
    \end{align*}for another constant \(\kappa_{2}\neq0\). Finally, a suitable renormalization of \(z_{2,1}\) and \(z_{4,1}\) ensures that \(\kappa_{1}=\kappa_{2}=1\). 
\end{proof}    

\begin{proof}[Proof of Proposition \ref{fouriermode1prop}]
    It suffices to solve equation (\ref{fmode1system}) where \(\varphi=\big(\psi_{11}^{1},\psi_{12}^{1}\big)\) and \(\widetilde{h}=\big(h_{11}^{1},h_{12}^{1}\big)\). 
    By Proposition \ref{mode1repformula}, a particular solution of this problem can be built through the following formula
    \begin{equation}
    \label{mode1partsol}
    \begin{split}
    \begin{pmatrix}[1.3]
        \psi_{11}^{1}(r)\\
        \psi_{12}^{1}(r)
    \end{pmatrix}=&\quad\left(\int_{0}^{r} w(s)^{2}s\:\widetilde{h}(s)\cdot z_{2,1}(s)\:ds\right)z_{1,1}(r)\\
    &-\left(\int_{0}^{r} w(s)^{2}s\:\widetilde{h}(s)\cdot z_{1,1}(s)\:ds\right)z_{2,1}(r)\\
    &-\left(\int_{r}^{\infty} w(s)^{2}s\:\widetilde{h}(s)\cdot z_{4,1}(s)\:ds\right)z_{3,1}(r)\\
    &-\left(\int_{0}^{r} w(s)^{2}s\:\widetilde{h}(s)\cdot z_{3,1}(s)\:ds\right)z_{4,1}(r).
\end{split}
\end{equation}It is then directly verified, using (\ref{mode1kernelasymprop}), that the estimate 
\begin{equation}
\label{mode1partsolest0}
    \abs{\psi_{11}^{1}(r)}+\abs{\psi_{12}^{1}(r)}\:\leq\: Cr\abs{\log r}\:\norm{h}_{**}
\end{equation}holds for \(r\leq 1/2\), where \(h\) is given by (\ref{hmode1v1}). Concerning the behaviour of (\ref{mode1partsol}) as \(r\to\infty\), note that the orthogonality condition \(\left\langle h,\frac{\partial W}{\partial x_{2}} \right\rangle=0\) reads, in terms of \(\widetilde{h}\), as
\begin{equation*}
    \int_{0}^{\infty}w(s)^{2}s\:\widetilde{h}(s)\cdot z_{1,1}(s)\:ds=0.
\end{equation*}This allows us to exchange the integral in the second term of (\ref{mode1partsol}) with an integral from \(r\) to \(\infty\), and it is then straightforward to check that the estimate
\begin{equation}
\label{mode1partsolestinf}
    \abs{\psi_{11}^{1}(r)}+\abs{\psi_{12}^{1}(r)}\:\leq\: C\norm{h}_{**}
\end{equation}holds for \(r\geq1/2\). The proof of the proposition is complete.
\end{proof}

\begin{proof}[Proof of Corollary \ref{fouriermode1cor}] For \(h\) of the form (\ref{hmode1v2}) and \(\phi\) of the form (\ref{phimode1v2}), problem (\ref{linearizedeq}) reduces to the ODE system (\ref{fmode1system}) where \(\varphi=\big(\psi_{11}^{2},-\psi_{12}^{2}\big)\) and \(\widetilde{h}=\big(h_{11}^{2},-h_{12}^{2}\big)\). Moreover, the orthogonality condition \(\left\langle h,\frac{\partial W}{\partial x_{1}} \right\rangle=0\) for (\ref{hmode1v2}) reads
\begin{equation*}
    \int_{0}^{\infty}w(s)^{2}s\:\widetilde{h}(s)\cdot z_{1,1}(s)\:ds=0,\quad\quad\widetilde{h}=\big(h_{11}^{2},-h_{12}^{2}\big).
\end{equation*}The result then follows by applying the representation formula (\ref{mode1partsol}) to \((\psi_{11}^{2},-\psi_{12}^{2}\big)\) with right-hand side \(\widetilde{h}=\big(h_{11}^{2},-h_{12}^{2}\big)\).
\end{proof}

\begin{remark}
    The proofs above show that, in the Fourier mode \(k=1\) case, we actually have a stronger estimate than (\ref{mainestimate}). In particular, we have estimates (\ref{mode1partsolest0}), (\ref{mode1partsolestinf}) and similar for \(\psi_{11}^{2}\), \(\psi_{12}^{2}\). If \(\widetilde{h}=\big(h_{11}^{1},h_{12}^{1}\big)\) is not assumed to satisfy any orthogonality condition, it is directly verified that (\ref{mode1partsolestinf}) is replaced by the estimate
    \begin{equation*}
       \abs{\psi_{11}^{1}(r)}\leq Cr\norm{h}_{**},\quad\abs{\psi_{12}^{1}(r)}\leq C\norm{h}_{**}
    \end{equation*}for \(r\geq1/2\). In a similar manner, we have the same estimate for \((\psi_{11}^{2},\psi_{12}^{2}\big)\) if \(\big(h_{11}^{2},h_{12}^{2}\big)\) does not satisfy any orthogonality condition. This justifies the weaker estimate stated in Remark \ref{noorthogconds} following the main theorem.
\end{remark}

\section{Fourier modes \texorpdfstring{\(k\geq2\)}{k>=2}}

The aim of this section is to prove Theorem \ref{mainthm} for each Fourier mode \(k\geq2\). We have the following results.

\begin{proposition}
\label{fouriermodegeq2prop}
Let \(k\geq2\). There exists a constant \(C_{k}>0\) such that for any 
\begin{equation}
\label{hmodegeq2v1}
    h=iWh_{k}^{1},\quad\quad h_{k}^{1}=h_{k1}^{1}(r)\cos k\theta+ih_{k2}^{1}(r)\sin k\theta,
\end{equation}
satisfying \(\norm{h}_{**}<\infty\), equation (\ref{linearizedeq}) has a solution 
\begin{equation}
\label{phimodegeq2v1}
    \phi=iW\psi_{k}^{1},\quad\quad \psi_{k}^{1}=\psi_{k1}^{1}(r)\cos k\theta+i\psi_{k2}^{1}(r)\sin k\theta,
\end{equation}such that 
\begin{equation}
\label{fouriermodegeq2estv1}
    \norm{\phi}_{*}\leq C_{k}\norm{h}_{**}.
\end{equation}
\end{proposition}

\begin{corollary}
\label{fouriermodegeq2cor}
Let \(k\geq2\). There exists a constant \(C_{k}>0\) such that for any
    \begin{equation}
\label{hmodegeq2v2}
    h=iWh_{k}^{2},\quad\quad h_{k}^{2}=h_{k1}^{2}(r)\sin k\theta+ih_{k2}^{2}(r)\cos k\theta,
\end{equation}
satisfying \(\norm{h}_{**}<\infty\), equation (\ref{linearizedeq}) has a solution 
\begin{equation}
\label{phimodegeq2v2}
    \phi=iW\psi_{k}^{2},\quad\quad \psi_{k}^{2}=\psi_{k1}^{2}(r)\sin k\theta+i\psi_{k2}^{2}(r)\cos k\theta,
\end{equation}such that 
\begin{equation}
\label{fouriermodegeq2estv2}
    \norm{\phi}_{*}\leq C_{k}\norm{h}_{**}.
\end{equation}
\end{corollary}

As in the Fourier mode \(k=1\) case, it suffices to consider the setting of Proposition \ref{fouriermodegeq2prop}. The proof of Corollary \ref{fouriermodegeq2cor} then follows using a symmetry argument. Recall that, for \(h\) of the form (\ref{hmodegeq2v1}) and \(\phi\) of the form (\ref{phimodegeq2v1}), problem (\ref{linearizedeq}) reduces to the coupled ODE system
\begin{equation}
\label{fmodegeq2system}
\varphi''+\left(2\frac{w'}{w}+\frac{1}{r}\right)\varphi'-\frac{1}{r^{2}}\begin{pmatrix}
        k^{2} & 2k \\
        2k & k^{2}+2w^{2}r^{2} 
    \end{pmatrix}\varphi=\widetilde{h},
\end{equation}where 
\begin{equation}
    \varphi=\begin{pmatrix}[1.3]
        \psi_{k1}^{1}\\
        \psi_{k2}^{1}
    \end{pmatrix},\quad\quad\widetilde{h}= 
    \begin{pmatrix}[1.3]
        h_{k1}^{1}\\
        h_{k2}^{1}
    \end{pmatrix}.
\end{equation}The analysis of the above problem proceeds by generalizing the results of the previous section. In particular, we first classify the solutions to the homogeneous version of (\ref{fmodegeq2system}) in Subsection \ref{solhomproblemfmodegeq2}. The general solution to the inhomogeneous problem is then presented in Subsection \ref{solinhomproblemfmodegeq2}.

\subsection{Solutions of the homogeneous problem: Fourier modes \texorpdfstring{\(k\geq2\)}{k>=2}}
\label{solhomproblemfmodegeq2}

The homogeneous problem associated to (\ref{fmodegeq2system}) reads 
\begin{equation}
\label{fmodegeq2systemhom}
\varphi''+\left(2\frac{w'}{w}+\frac{1}{r}\right)\varphi'-\frac{1}{r^{2}}\begin{pmatrix}
        k^{2} & 2k \\
        2k & k^{2}+2w^{2}r^{2} 
    \end{pmatrix}\varphi=0,
\end{equation}where \(\varphi=(\varphi_{1},\varphi_{2})\). We have the following nondegeneracy result in this setting. 

\begin{lemma}
\label{fmodegeq2nondegeneracy}
    Let \(k\geq2\). The only solution of equation (\ref{fmodegeq2systemhom}) which blows up at most like \(O(r^{-1})\) as \(r\to0^{+}\) and remains bounded as \(r\to\infty\) is the trivial solution \(\varphi\equiv0\).
\end{lemma}

\begin{proof}
By the nondegeneracy result of Pacard and Rivi\`{e}re \cite{pacardriviere2000}*{Theorem 3.2}, there are no bounded nontrivial solutions of \(L[\phi]=0\) of the form (\ref{phimodegeq2v1}) for \(k\geq2\). Translating this fact to (\ref{fmodegeq2systemhom}), and using the properties \(\abs{W}=w\sim r\) as \(r\to0^{+}\) and \(\abs{W}=w\sim1\) as \(r\to\infty\), we see that the statement of the lemma holds.
\end{proof}

The aim of this subsection is to prove the following result.

\begin{proposition}
\label{modegeq2kernel}
Let \(k\geq2\). The homogeneous problem (\ref{fmodegeq2systemhom}) 
admits four linearly independent solutions \(z_{1,k}\), \(z_{2,k}\), \(z_{3,k}\), \(z_{4,k}\) with the asymptotic behaviour
\begin{equation}
\label{modegeq2kernelasymprop}
\begin{aligned}[c]
z_{1,k}(r)&=\begin{pmatrix}[1.3]
    O(r^{-k})\\
    O(r^{-k})
    \end{pmatrix},\text{ as }r\to0^{+},\\
z_{2,k}(r)&=\begin{pmatrix}[1.3]
    O(r^{k-2})\\
    O(r^{k-2})
    \end{pmatrix},\text{ as }r\to0^{+},\\
z_{3,k}(r)&=\begin{pmatrix}[1.3]
    O(r^{k})\\
    O(r^{k})
    \end{pmatrix},\text{ as }r\to0^{+},\\
z_{4,k}(r)&=\begin{pmatrix}[1.3]
    O(r^{-2-k})\\
    O(r^{-2-k})
    \end{pmatrix},\text{ as }r\to0^{+},    
\end{aligned}
\quad\quad
\begin{aligned}[c]
z_{1,k}(r)&=\begin{pmatrix}[1.3]
    O(r^{-k})\\
    O(r^{-k-2})
    \end{pmatrix},\text{ as }r\to\infty,\\
z_{2,k}(r)&=\begin{pmatrix}[1.3]
    O(r^{k})\\
    O(r^{k-2})
    \end{pmatrix},\text{ as }r\to\infty,\\
z_{3,k}(r)&=\begin{pmatrix}[1.3]
    O\big(r^{-5/2}e^{\sqrt{2}r}\big)\\
    O\big(r^{-1/2}e^{\sqrt{2}r}\big)
    \end{pmatrix},\text{ as }r\to\infty,\\
z_{4,k}(r)&=\begin{pmatrix}[1.3]
    O\big(r^{-5/2}e^{-\sqrt{2}r}\big)\\
    O\big(r^{-1/2}e^{-\sqrt{2}r}\big)
    \end{pmatrix},\text{ as }r\to\infty.    
\end{aligned}
\end{equation}
\end{proposition}

As in the Fourier mode \(k=1\) case, we first need to classify the possible behaviour of solutions as \(r\to0^{+}\) and as \(r\to\infty\).

\begin{lemma}
\label{modegeq2kernelat0}
   For \(k\geq2\), the solution space of equation (\ref{fmodegeq2systemhom}) is spanned, for small \(r>0\), by four functions with the asymptotic behaviour
    \begin{equation}
    \label{modegeq2kernelat0eq}
    \begin{aligned}[c]
    \begin{pmatrix}[1.3]
        r^{-k}+O(r^{-k+2})\\
        -r^{-k}+O(r^{-k+2})
    \end{pmatrix},\text{ as }r\to0^{+},\\
    \begin{pmatrix}[1.3]
        r^{k}+O(r^{k+2})\\
        r^{k}+O(r^{k+2})
    \end{pmatrix},\text{ as }r\to0^{+},
    \end{aligned}\quad\quad
    \begin{aligned}[c]
    \begin{pmatrix}[1.3]
        r^{k-2}+O(r^{k})\\
        -r^{k-2}+O(r^{k})
    \end{pmatrix},\text{ as }r\to0^{+},\\
    \begin{pmatrix}[1.3]
        r^{-2-k}+O(r^{-k})\\
        r^{-2-k}+O(r^{-k})
    \end{pmatrix},\text{ as }r\to0^{+}. 
    \end{aligned}
\end{equation}
\end{lemma}

\begin{proof}
As \(r\to0^{+}\), problem (\ref{fmodegeq2systemhom}) becomes asymptotic to the equation
\begin{equation}
\label{fmodegeq2systemhomrto0}
\varphi''+\frac{3}{r}\varphi'-\frac{1}{r^{2}}\begin{pmatrix}
        k^{2} & 2k \\
        2k & k^{2} 
    \end{pmatrix}\varphi=0.
\end{equation}The four linearly independent solutions of (\ref{fmodegeq2systemhomrto0}) are given by
\begin{equation*}
\begin{pmatrix}[1.3]
        r^{-k}\\
        -r^{-k}
    \end{pmatrix},\quad 
\begin{pmatrix}[1.3]
        r^{k-2}\\
        -r^{k-2}
    \end{pmatrix},\quad
\begin{pmatrix}[1.3]
        r^{k}\\
        r^{k}
    \end{pmatrix},\quad 
\begin{pmatrix}[1.3]
        r^{-2-k}\\
        r^{-2-k}
    \end{pmatrix},     
\end{equation*}as follows from direct calculation. A contraction mapping argument then gives four linearly independent solutions of (\ref{fmodegeq2systemhom}) for small \(r>0\) which have the properties (\ref{modegeq2kernelat0eq}).
\end{proof}

\begin{lemma}
\label{modegeq2kernelatinf}
    For \(k\geq2\), the solution space of equation (\ref{fmodegeq2systemhom}) is spanned, for large \(r>0\), by four functions with the asymptotic behaviour
    \begin{equation}
    \begin{aligned}[c]
    \begin{pmatrix}[1.3]
        r^{-k}+O(r^{-k-2})\\
        -kr^{-k-2}+O(r^{-k-4})
    \end{pmatrix},\text{ as }r\to\infty,\\
    \begin{pmatrix}[1.3]
        r^{k}+O(r^{k-2})\\
        -kr^{k-2}+O(r^{k-4})
    \end{pmatrix},\text{ as }r\to\infty,\\
    \begin{pmatrix}[1.3]
        kr^{-5/2}e^{\sqrt{2}r}+O\big(r^{-7/2}e^{\sqrt{2}r}\big)\\
        r^{-1/2}e^{\sqrt{2}r}+O\big(r^{-3/2}e^{\sqrt{2}r}\big)
    \end{pmatrix},\text{ as }r\to\infty,\\
    \begin{pmatrix}[1.3]
        kr^{-5/2}e^{-\sqrt{2}r}+O\big(r^{-7/2}e^{-\sqrt{2}r}\big)\\
        r^{-1/2}e^{-\sqrt{2}r}+O\big(r^{-3/2}e^{-\sqrt{2}r}\big)
    \end{pmatrix},\text{ as }r\to\infty.
    \end{aligned}
    \end{equation}

\begin{proof}
    As \(r\to\infty\), problem (\ref{fmodegeq2systemhom}) becomes asymptotic to the coupled system of equations
    \begin{align}
        &\varphi_{1}''+\frac{1}{r}\varphi_{1}'-\frac{k^{2}}{r^{2}}\varphi_{1}-\frac{2k}{r^{2}}\varphi_{2}=0,\label{proofrtoinfmodegeq2hom1}\\
        &\varphi_{2}''+\frac{1}{r}\varphi_{2}'-\frac{2k}{r^{2}}\varphi_{1}-2\varphi_{2}=0.\label{proofrtoinfmodegeq2hom2}
    \end{align}Neglecting the \(\varphi_{2}\) term in (\ref{proofrtoinfmodegeq2hom1}), we arrive at the equation \(\varphi_{1}''+\tfrac{1}{r}\varphi_{1}'-\tfrac{k^{2}}{r^{2}}\varphi_{1}=0\), whose solutions are spanned by the functions \(\varphi_{1}=r^{-k}\) and \(\varphi_{1}=r^{k}\). Substituting this ansatz for \(\varphi_{1}\) into (\ref{proofrtoinfmodegeq2hom1})-(\ref{proofrtoinfmodegeq2hom2}) and matching powers of \(r\) for \(\varphi_{2}\), we can construct two formal power series solutions to this system with asymptotic behaviour
    \begin{equation*}
        \varphi= \begin{pmatrix}[1.3]
        r^{-k}+O(r^{-k-2})\\
        -kr^{-k-2}+O(r^{-k-4})
    \end{pmatrix},\quad\quad\varphi=\begin{pmatrix}[1.3]
        r^{k}+O(r^{k-2})\\
        -kr^{k-2}+O(r^{k-4})
    \end{pmatrix},
    \end{equation*}as \(r\to\infty\). A contraction mapping argument then gives two linearly independent solutions of the original equation (\ref{fmodegeq2systemhom}) for large \(r>0\) with these properties.

    Turning to the second equation (\ref{proofrtoinfmodegeq2hom2}) and neglecting the \(\varphi_{1}\) term, we arrive once more at the Bessel-type problem \(\varphi_{2}''+\tfrac{1}{r}\varphi_{2}'-2\varphi_{2}=0\). Substituting the ansatz \(\varphi_{2}\sim r^{-1/2}e^{\sqrt{2}r}\) and \(\varphi_{2}\sim r^{-1/2}e^{-\sqrt{2}r}\) into (\ref{proofrtoinfmodegeq2hom1})-(\ref{proofrtoinfmodegeq2hom2}) and matching powers of \(r\) for \(\varphi_{1}\), we find two generalized formal power series solutions to this system with asymptotic behaviour 
    \begin{equation*}
        \varphi=\begin{pmatrix}[1.3]
        kr^{-5/2}e^{\sqrt{2}r}+O\big(r^{-7/2}e^{\sqrt{2}r}\big)\\
        r^{-1/2}e^{\sqrt{2}r}+O\big(r^{-3/2}e^{\sqrt{2}r}\big)
    \end{pmatrix},\quad\quad\varphi=\begin{pmatrix}[1.3]
       k r^{-5/2}e^{-\sqrt{2}r}+O\big(r^{-7/2}e^{-\sqrt{2}r}\big)\\
        r^{-1/2}e^{-\sqrt{2}r}+O\big(r^{-3/2}e^{-\sqrt{2}r}\big)
    \end{pmatrix},
    \end{equation*}as \(r\to\infty\). Again a contraction mapping argument can be employed to obtain two linearly independent solutions of (\ref{fmodegeq2systemhom}) for large \(r>0\) with these properties.
\end{proof}
\end{lemma}

We also need the following lemma.

\begin{lemma}
\label{nosignchangemodegeq2}
    Let \(\varphi=(\varphi_{1},\varphi_{2})\) be a solution of (\ref{fmodegeq2systemhom}).
    \begin{enumerate}[(i)]
        \item If \(\varphi_{1},\varphi_{2}\to0\) as \(r\to0^{+}\) and \(\varphi_{1}(r),\varphi_{2}(r)>0\) for small \(r>0\), then \(\varphi_{1}(r),\varphi_{2}(r)>0\) for all \(r\in(0,\infty)\).
        \item If \(\varphi_{1},\varphi_{2}\to0\) as \(r\to\infty\) and \(\varphi_{1}(r),\varphi_{2}(r)>0\) for large \(r>0\), then \(\varphi_{1}(r),\varphi_{2}(r)>0\) for all \(r\in(0,\infty)\).
    \end{enumerate}
\end{lemma}

\begin{proof}
    The proof is essentially identical to that of Lemma \ref{nosignchangemode1} in the mode \(k=1\) case. We omit the details. 
\end{proof}

Combining the above results, we can complete the classification of solutions for the homogeneous problem.

\begin{proof}[Proof of Proposition \ref{modegeq2kernel}]
    Let \(z_{3,k}(r)\) be a solution of the homogeneous problem from Lemma \ref{modegeq2kernelat0} with the behaviour
    \begin{equation*}
    z_{3,k}(r)=\begin{pmatrix}[1.3]
        r^{k}+O(r^{k+2})\\
        r^{k}+O(r^{k+2})
    \end{pmatrix}    
    \end{equation*}for small \(r>0\). By standard ODE theory \(z_{3,k}(r)\) extends to a solution of (\ref{fmodegeq2systemhom}) defined for all \(r\in(0,\infty)\), and Lemma \ref{fmodegeq2nondegeneracy} implies that \(z_{3,k}(r)\) must become unbounded as \(r\to\infty\). By Lemma \ref{nosignchangemodegeq2}, the function \(z_{3,k}(r)\) must have components of the same sign as \(r\to\infty\). It follows from Lemma \ref{modegeq2kernelatinf} that
    \begin{equation*}
        z_{3,k}(r)=c\begin{pmatrix}[1.3]
        kr^{-5/2}e^{\sqrt{2}r}+O\big(r^{-7/2}e^{\sqrt{2}r}\big)\\
        r^{-1/2}e^{\sqrt{2}r}+O\big(r^{-3/2}e^{\sqrt{2}r}\big)
    \end{pmatrix},\text{ as }r\to\infty,
    \end{equation*}for some constant \(c>0\).

    Now consider a solution \(z_{4,k}(r)\) of the homogeneous problem from Lemma \ref{modegeq2kernelatinf} with the behaviour
    \begin{equation*}
    z_{4,k}(r)=\begin{pmatrix}[1.3]
        kr^{-5/2}e^{-\sqrt{2}r}+O\big(r^{-7/2}e^{-\sqrt{2}r}\big)\\
        r^{-1/2}e^{-\sqrt{2}r}+O\big(r^{-3/2}e^{-\sqrt{2}r}\big)
    \end{pmatrix},\text{ as }r\to\infty.
    \end{equation*}By Lemma \ref{fmodegeq2nondegeneracy} and Lemma \ref{nosignchangemodegeq2}, the solution \(z_{4,k}(r)\) must blow up faster than \(O(r^{-1})\) and have components of the same sign as \(r\to0^{+}\). It follows from Lemma \ref{modegeq2kernelat0} that 
    \begin{equation*}
     z_{4,k}(r)=c\begin{pmatrix}[1.3]
        r^{-2-k}+O(r^{-k})\\
        r^{-2-k}+O(r^{-k})
    \end{pmatrix},\text{ as }r\to0^{+}.   
    \end{equation*}Next, let \(z_{1,k}(r)\) denote a solution of the homogeneous problem from Lemma \ref{modegeq2kernelatinf} with the property
    \begin{equation*}
        z_{1,k}(r)=\begin{pmatrix}[1.3]
        r^{-k}+O(r^{-k-2})\\
        -kr^{-k-2}+O(r^{-k-4})
    \end{pmatrix},\text{ as }r\to\infty.
    \end{equation*}By Lemma \ref{fmodegeq2nondegeneracy}, the function \(z_{1,k}(r)\) must blow up faster than \(O(r^{-1})\) as \(r\to0^{+}\). Subtracting a multiple of \(z_{4,k}(r)\), if necessary, we can arrange using Lemma \ref{modegeq2kernelat0} that the behaviour of \(z_{1,k}(r)\) as \(r\to0^{+}\) is given by
    \begin{equation*}
        z_{1,k}(r)=c\begin{pmatrix}[1.3]
        r^{-k}+O(r^{-k+2})\\
        -r^{-k}+O(r^{-k+2})
    \end{pmatrix}.
    \end{equation*}Finally, let \(z_{2,k}(r)\) be a solution of the homogeneous problem from Lemma \ref{modegeq2kernelat0} with the behaviour
    \begin{equation*}
    z_{2,k}(r)=\begin{pmatrix}[1.3]
        r^{k-2}+O(r^{k})\\
        -r^{k-2}+O(r^{k})
    \end{pmatrix},\text{ as }r\to0^{+}.
    \end{equation*}By Lemma \ref{fmodegeq2nondegeneracy}, the function \(z_{2,k}(r)\) must become unbounded as \(r\to\infty\). Moreover, the function \(z_{2,k}(r)\) must have the remaining behaviour at infinity described in Lemma \ref{modegeq2kernelatinf}. Thus 
    \begin{equation*}
        z_{2,k}(r)=c\begin{pmatrix}[1.3]
        r^{k}+O(r^{k-2})\\
        -kr^{k-2}+O(r^{k-4})
    \end{pmatrix},\text{ as }r\to\infty.
    \end{equation*}
\end{proof}

\subsection{Solutions of the inhomogeneous problem: Fourier modes \texorpdfstring{\(k\geq2\)}{k>=2}} 
\label{solinhomproblemfmodegeq2}

In this subsection we obtain a representation formula for solutions of the inhomogeneous problem (\ref{fmodegeq2system}), and complete the proof of Proposition \ref{fouriermodegeq2prop} and Corollary \ref{fouriermodegeq2cor}.

\begin{proposition}
 \label{modegeq2repformula}  
    Let \(\varphi=(\varphi_{1},\varphi_{2})\) and \(\widetilde{h}=(\widetilde{h}_{1},\widetilde{h}_{2})\). For \(k\geq2\), the general solution of the inhomogeneous problem (\ref{fmodegeq2system}) can be written in the form
    \begin{equation*}
\begin{split}
    \varphi(r)=&\quad\left(\int w(s)^{2}s\:\widetilde{h}(s)\cdot z_{2,k}(s)\:ds\right)z_{1,k}(r)\\
    &-\left(\int w(s)^{2}s\:\widetilde{h}(s)\cdot z_{1,k}(s)\:ds\right)z_{2,k}(r)\\
    &+\left(\int w(s)^{2}s\:\widetilde{h}(s)\cdot z_{4,k}(s)\:ds\right)z_{3,k}(r)\\
    &-\left(\int w(s)^{2}s\:\widetilde{h}(s)\cdot z_{3,k}(s)\:ds\right)z_{4,k}(r),
\end{split}
\end{equation*}where \(z_{1,k},\ldots, z_{4,k}\) denote the four linearly independent solutions of the homogeneous problem described in Proposition \ref{modegeq2kernel}, and the symbols \(\int\) denote arbitrary antiderivatives.
\end{proposition}

\begin{proof}
    The proof follows along exactly the same lines as the mode \(k=1\) case (Proposition \ref{mode1repformula}). We omit the details.
\end{proof}

\begin{proof}[Proof of Proposition \ref{fouriermodegeq2prop}]
    It suffices to solve equation (\ref{fmodegeq2system}) where \(\varphi=\big(\psi_{k1}^{1},\psi_{k2}^{1}\big)\) and \(\widetilde{h}=\big(h_{k1}^{1},h_{k2}^{1}\big)\). By Proposition \ref{modegeq2repformula}, a solution of this problem can be constructed through the following formula
    \begin{equation}
    \label{modegeq2partsol}
    \begin{split}
    \begin{pmatrix}[1.3]
        \psi_{k1}^{1}(r)\\
        \psi_{k2}^{1}(r)
    \end{pmatrix}=&\quad\left(\int_{0}^{r} w(s)^{2}s\:\widetilde{h}(s)\cdot z_{2,k}(s)\:ds\right)z_{1,k}(r)\\
    &+\left(\int_{r}^{\infty} w(s)^{2}s\:\widetilde{h}(s)\cdot z_{1,k}(s)\:ds\right)z_{2,k}(r)\\
    &-\left(\int_{r}^{\infty} w(s)^{2}s\:\widetilde{h}(s)\cdot z_{4,k}(s)\:ds\right)z_{3,k}(r)\\
    &-\left(\int_{0}^{r} w(s)^{2}s\:\widetilde{h}(s)\cdot z_{3,k}(s)\:ds\right)z_{4,k}(r).
\end{split}
\end{equation}It is then directly verified, using (\ref{modegeq2kernelasymprop}), that the estimate \begin{equation}
\label{modegeq2partsolest0}
    \abs{\psi_{k1}^{1}(r)}+\abs{\psi_{k2}^{1}(r)}\:\leq\:
    \left\{
    \begin{array}{ll}
    C\norm{h}_{**}& \text{ if }k=2, \\
    Cr\abs{\log r}\:\norm{h}_{**}& \text{ if }k=3,\\
    C_{k}r\norm{h}_{**}& \text{ if }k\geq4,
    \end{array}\right.
\end{equation}holds for \(r\leq 1/2\), where \(h\) is given by (\ref{hmodegeq2v1}). Concerning the behaviour of (\ref{modegeq2partsol}) as \(r\to\infty\), a similar verification gives the estimate
\begin{equation}
\label{modegeq2partsolestinf}
    \abs{\psi_{k1}^{1}(r)}+\abs{\psi_{k2}^{1}(r)}\:\leq\: C_{k}\norm{h}_{**}
\end{equation} for \(r\geq1/2\). The proof of the proposition is complete.
\end{proof}

\begin{proof}[Proof of Corollary \ref{fouriermodegeq2cor}]
   For \(h\) of the form (\ref{hmodegeq2v2}) and \(\phi\) of the form (\ref{phimodegeq2v2}), problem (\ref{linearizedeq}) reduces to the ODE system (\ref{fmodegeq2system}) where \(\varphi=\big(\psi_{k1}^{2},-\psi_{k2}^{2}\big)\) and \(\widetilde{h}=\big(h_{k1}^{2},-h_{k2}^{2}\big)\). The result then follows by applying the representation formula (\ref{modegeq2partsol}) to \((\psi_{k1}^{2},-\psi_{k2}^{2}\big)\) with right-hand side \(\widetilde{h}=\big(h_{k1}^{2},-h_{k2}^{2}\big)\). 
\end{proof}

\begin{remark}
    The proofs above show that, in the Fourier mode \(k\geq2\) case, we actually have a stronger estimate than (\ref{mainestimate}). In particular, we have estimates (\ref{modegeq2partsolest0}), (\ref{modegeq2partsolestinf}) and similar for \(\psi_{k1}^{2}\), \(\psi_{k2}^{2}\).
\end{remark}

\section{Summing the Fourier modes}

Combining the results of the previous sections, we see that the operator \(L\) can be inverted in our weighted \(L^{\infty}\) spaces for right-hand sides \(h\) containing finitely many non-trivial modes in their Fourier expansion. More precisely, if \(m>0\) is an integer and \(h\) decomposes like (\ref{hfourier}) with the orthogonality conditions and \(h_{k}^{1}=h_{k}^{2}=0\) for all \(k\geq m\), then a solution \(\phi\) of (\ref{linearizedeq}) exists with the estimate \(\norm{\phi}_{*}\leq C_{m}\norm{h}_{**}\). To complete the proof of Theorem \ref{mainthm}, we need to establish a uniform estimate for the inverse which is valid after summing all Fourier modes. This is achieved through the following result.

\begin{proposition}
\label{fouriersumprop}
    There exists a constant \(C>0\) such that for any 
\begin{align}
\label{hinfsumgeq3}
    \begin{split}
    h&=iW\left(\sum_{k=3}^{\infty}h_{k}^{1}+\sum_{k=3}^{\infty}h_{k}^{2}\right),\\
    h_{k}^{1}&=h_{k1}^{1}(r)\cos{k\theta}+ih_{k2}^{1}(r)\sin{k\theta},\\
    h_{k}^{2}&=h_{k1}^{2}(r)\sin{k\theta}+ih_{k2}^{2}(r)\cos{k\theta},
    \end{split}
\end{align}
    satisfying \(\norm{h}_{**}<\infty\), equation (\ref{linearizedeq}) has a solution \(\phi=iW\psi\),
    \begin{align}
    \label{psiinfsumgeq3}
    \begin{split}
    \psi&=\sum_{k=3}^{\infty}\psi_{k}^{1}+\sum_{k=3}^{\infty}\psi_{k}^{2},\\
    \psi_{k}^{1}&=\psi_{k1}^{1}(r)\cos{k\theta}+i\psi_{k2}^{1}(r)\sin{k\theta},\\
    \psi_{k}^{2}&=\psi_{k1}^{2}(r)\sin{k\theta}+i\psi_{k2}^{2}(r)\cos{k\theta},
    \end{split}
\end{align}
    
    such that
    \begin{equation}
    \label{infsumest}
        \sup_{\mathbb{R}^{2}}\:\abs{\psi}\leq C\norm{h}_{**}.
    \end{equation}For each \(k\geq 3\), the pair \((\psi_{k1}^{1},\psi_{k2}^{1})\) can be represented through formula (\ref{modegeq2partsol}) with right-hand side \(\widetilde{h}=(h_{k1}^{1},h_{k2}^{1})\). Similarly for \((\psi_{k1}^{2},-\psi_{k2}^{2})\) with right-hand side \(\widetilde{h}=(h_{k1}^{2},-h_{k2}^{2})\).
\end{proposition}

Note that the above result concerns functions supported in Fourier modes \(k\geq3\): the statement of Theorem \ref{mainthm} then follows by combining with Proposition \ref{fouriermode0prop} in the mode \(k=0\) case, Proposition \ref{fouriermode1prop} and Corollary \ref{fouriermode1cor} in the mode \(k=1\) case, and Proposition \ref{fouriermodegeq2prop} and Corollary \ref{fouriermodegeq2cor} in the mode \(k=2\) case.

To prove Proposition \ref{fouriersumprop}, we will need some maximum principle type estimates for the operator \(L\) on bounded domains. Similar estimates have appeared before in the work of Pacard and Rivi\`{e}re \cite{pacardriviere2000}*{Section 4}, nevertheless we repeat the arguments here for the reader's convenience. Recalling equation (\ref{linearizedeqpsi}), we note that the problem \(L[\phi]=h\) can be written in terms of \(h=iW\big(h_{1}+ih_{2}\big)\) and \(\phi=iW\big(\psi_{1}+i\psi_{2}\big)\) as the coupled system
\begin{align}
    \Delta\psi_{1}+2\frac{w'}{w}\partial_{r}\psi_{1}-\frac{2}{r^{2}}\partial_{\theta}\psi_{2}&=h_{1},\label{linearizedpsieq1}\\
    \Delta\psi_{2}+2\frac{w'}{w}\partial_{r}\psi_{2}+\frac{2}{r^{2}}\partial_{\theta}\psi_{1}-2w^{2}\psi_{2}&=h_{2}.\label{linearizedpsieq2}
\end{align}We intend to estimate \(\psi_{1}\) and \(\psi_{2}\) through a separate analysis of the two equations above, by first neglecting the \(\partial_{\theta}\) terms. Concerning (\ref{linearizedpsieq1}), we have the following lemma.

\begin{lemma}
\label{psi1lemma}
    Let \(0<\rho<R\), and let \(g_{1}\), \(h_{1}\) be bounded functions such that \(h_{1}\) contains no \(k=0\) mode in its Fourier expansion. There exists a constant \(C>0\), independent of \(\rho\) and \(R\), such that the unique solution of the problem 
    \begin{align}
        \Delta\psi_{1}+2\frac{w'}{w}\partial_{r}\psi_{1}&=h_{1},\quad\text{in }B_{R}\backslash B_{\rho},\label{psi1inhom}\\
        \psi_{1}&=g_{1},\quad\text{on }\partial B_{\rho},\label{psi1boundarybrho}\\
        \psi_{1}&=0,\quad\text{ on }\partial B_{R},\label{psi1boundarybR}
    \end{align}satisfies the estimate
    \begin{equation}
    \label{psi1lemmaest}
        \sup_{B_{R}\backslash B_{\rho}}\:\abs{\psi_{1}}\leq C\bigg(\sup_{\partial B_{\rho}}\:\abs{g_{1}}+\sup_{B_{R}\backslash B_{\rho}}\:\abs{r^{2}h_{1}}\bigg).
    \end{equation}
\end{lemma}

\begin{proof}
    First suppose that \(g_{1}=0\). Writing \(\psi_{1}\) and \(h_{1}\) in terms of their Fourier series \(\psi_{1}=\sum_{\abs{k}\geq1}\hat{\psi}_{1k}(r)e^{ik\theta}\) and \(h_{1}=\sum_{\abs{k}\geq1}\hat{h}_{1k}(r)e^{ik\theta}\), we see that problem (\ref{psi1inhom})-(\ref{psi1boundarybR}) reduces to the ODEs
    \begin{equation*}
        \hat{\psi}_{1k}''+\left(2\frac{w'}{w}+\frac{1}{r}\right)\hat{\psi}_{1k}'-\frac{k^{2}}{r^{2}}\hat{\psi}_{1k}=\hat{h}_{1k}
    \end{equation*}for \(\abs{k}\geq1\), with boundary conditions \(\hat{\psi}_{1k}(\rho)=\hat{\psi}_{1k}(R)=0\). A barrier argument for this equation gives the estimate \(\abs{\hat{\psi}_{1k}(r)}\leq Ck^{-2}\sup_{\rho<r<R}\:\abs{r^{2}\hat{h}_{1k}}\leq Ck^{-2}\sup_{\rho<r<R}\:\abs{r^{2}h_{1}}\), and summation over \(k\) then implies (\ref{psi1lemmaest}) with \(g_{1}=0\). In the general case, we can write \(\psi_{1}\) as a solution of the inhomogeneous problem (\ref{psi1inhom}) with zero boundary values plus a solution of the homogeneous problem with boundary conditions (\ref{psi1boundarybrho})-(\ref{psi1boundarybR}). The solution of the homogeneous problem is bounded by \(\sup_{\partial B_{\rho}}\:\abs{g_{1}}\) by the maximum principle. This completes the proof of the lemma.
\end{proof}

In a similar spirit, we have the following result for (\ref{linearizedpsieq2}) after dropping the \(\partial_{\theta}\psi_{1}\) term.

\begin{lemma}
\label{psi2lemma}
    Let \(\rho>0\) be sufficiently large, let \(R>\rho\), and let \(g_{2}\), \(h_{2}\) be bounded functions. There exists a constant \(C>0\), independent of \(\rho\) and \(R\), such that the unique solution of the problem 
    \begin{align}
        \Delta\psi_{2}+2\frac{w'}{w}\partial_{r}\psi_{2}-2w^{2}\psi_{2}&=h_{2},\quad\text{in }B_{R}\backslash B_{\rho},\label{psi2inhom}\\
        \psi_{2}&=g_{2},\quad\text{on }\partial B_{\rho},\label{psi2boundarybrho}\\
        \psi_{2}&=0,\quad\text{ on }\partial B_{R},\label{psi2boundarybR}
    \end{align}satisfies the estimate
    \begin{equation}
    \label{psi2lemmaest}
        \sup_{B_{R}\backslash B_{\rho}}\:\abs{\psi_{2}}\leq C\bigg(\sup_{\partial B_{\rho}}\:\abs{g_{2}}+\sup_{B_{R}\backslash B_{\rho}}\:\abs{h_{2}}\bigg).
    \end{equation}
\end{lemma}

\begin{proof}
    Since \(w(r)\to1\) as \(r\to\infty\), we can choose \(\rho>0\) large enough that \(2w(r)^{2}\geq1\) for all \(r\geq\rho\). Estimate (\ref{psi2lemmaest}) then follows using a suitable constant as a barrier.
\end{proof}

Using the function \(Cr^{-2}\) instead as a barrier, for a suitable constant \(C\), we obtain the following extension of the above result.

\begin{lemma}
\label{psi2lemmaext}
 Under the assumptions of Lemma \ref{psi2lemma}, the unique solution of (\ref{psi2inhom})-(\ref{psi2boundarybR}) with \(g_{2}=0\) satisfies the estimate
 \begin{equation}
     \sup_{B_{R}\backslash B_{\rho}}\:\abs{r^{2}\psi_{2}}\leq C\sup_{B_{R}\backslash B_{\rho}}\:\abs{r^{2}h_{2}},
 \end{equation}with \(C>0\) independent of \(\rho\) and \(R\).
\end{lemma}

The preceding lemmata can now be combined to prove the following estimate.

\begin{proposition}
\label{Lpsiouterest}
    There exist constants \(C\), \(\rho_{0}>0\) such that, for any \(\rho_{0}\leq\rho<R\) and any solution \(\psi=\psi_{1}+i\psi_{2}\) of (\ref{linearizedpsieq1})-(\ref{linearizedpsieq2}) in \(B_{R}\backslash B_{\rho}\) satisfying \(\psi_{1}=\psi_{2}=0\) on \(\partial B_{R}\) and containing no \(k=0\) mode in its Fourier expansion, we have the estimate
    \begin{equation}
    \label{psioutest}
        \sup_{B_{R}\backslash B_{\rho}}\abs{\psi}\leq C\bigg(\sup_{\partial B_{\rho}}\:\abs{\psi}+\sup_{B_{R}\backslash B_{\rho}}\abs{r^{2}h_{1}}+\sup_{B_{R}\backslash B_{\rho}}\abs{h_{2}}\bigg).
    \end{equation}
\end{proposition}

\begin{proof}
    Let \(\psi=\psi_{1}+i\psi_{2}\) be a solution of (\ref{linearizedpsieq1})-(\ref{linearizedpsieq2}) in \(B_{R}\backslash B_{\rho}\) satisfying \(\psi_{1}=\psi_{2}=0\) on \(\partial B_{R}\) and containing no \(k=0\) mode in its Fourier expansion. In view of (\ref{linearizedpsieq1}), we can decompose \(\psi_{1}\) as \(\psi_{1}=\widetilde{\psi}_{1}+\partial_{\theta}v_{1}\) where \(\widetilde{\psi}_{1}\) satisfies (\ref{psi1inhom})-(\ref{psi1boundarybR}) with \(g_{1}=\psi_{1}\), and \(v_{1}\) satisfies (\ref{psi1inhom})-(\ref{psi1boundarybR}) with \(h_{1}=(2\psi_{2})/r^{2}\), \(g_{1}=0\). By Lemma \ref{psi1lemma}, we have the estimates \(\abs{\widetilde{\psi}_{1}}\leq C\big(\sup_{\partial B_{\rho}}\:\abs{\psi_{1}}+\sup_{B_{R}\backslash B_{\rho}}\abs{r^{2}h_{1}}\big)\) and \(\abs{v_{1}}\leq C\sup_{B_{R}\backslash B_{\rho}}\abs{\psi_{2}}\). Moreover, standard elliptic theory implies \(\abs{\nabla v_{1}}\leq Cr^{-1}\sup_{B_{R}\backslash B_{\rho}}\abs{\psi_{2}}\), so \(\abs{\partial_{\theta}v_{1}}\leq C\sup_{B_{R}\backslash B_{\rho}}\abs{\psi_{2}}\). Thus we obtain
    \begin{equation}
    \label{psi1outest}
        \sup_{B_{R}\backslash B_{\rho}}\abs{\psi_{1}}\leq C\bigg(\sup_{\partial B_{\rho}}\:\abs{\psi_{1}}+\sup_{B_{R}\backslash B_{\rho}}\abs{r^{2}h_{1}}+\sup_{B_{R}\backslash B_{\rho}}\abs{\psi_{2}}\bigg).
    \end{equation}In a similar spirit, we can decompose \(\psi_{2}\) as \(\psi_{2}=\widetilde{\psi}_{2}+\partial_{\theta}v_{2}\) where \(\widetilde{\psi}_{2}\) satisfies (\ref{psi2inhom})-(\ref{psi2boundarybR}) with \(g_{2}=\psi_{2}\), and \(v_{2}\) satisfies (\ref{psi2inhom})-(\ref{psi2boundarybR}) with \(h_{2}=-(2\psi_{1})/r^{2}\), \(g_{2}=0\). We deduce from Lemma \ref{psi2lemma} and Lemma \ref{psi2lemmaext} the estimates \(\abs{\widetilde{\psi}_{2}}\leq C\big(\sup_{\partial B_{\rho}}\:\abs{\psi_{2}}+\sup_{B_{R}\backslash B_{\rho}}\abs{h_{2}}\big)\) and \(\abs{v_{2}}\leq Cr^{-2}\sup_{B_{R}\backslash B_{\rho}}\abs{\psi_{1}}\), and standard elliptic theory then implies \(\abs{\partial_{\theta}v_{2}}\leq r\abs{\nabla v_{2}}\leq Cr^{-1}\sup_{B_{R}\backslash B_{\rho}}\abs{\psi_{1}}\). Thus we obtain
    \begin{equation}
    \label{psi2outest}
        \sup_{B_{R}\backslash B_{\rho}}\abs{\psi_{2}}\leq C\bigg(\sup_{\partial B_{\rho}}\:\abs{\psi_{2}}+\sup_{B_{R}\backslash B_{\rho}}\abs{h_{2}}+\rho^{-1}\sup_{B_{R}\backslash B_{\rho}}\abs{\psi_{1}}\bigg).
    \end{equation}Provided \(\rho>0\) is chosen sufficiently large, we can then combine estimates (\ref{psi1outest}) and (\ref{psi2outest}) to deduce (\ref{psioutest}).
\end{proof}

A key consequence of the above result is the following estimate for \(L[\phi]=h\) on bounded domains.

\begin{proposition}
\label{Lestboundeddom}
    Let \(R>1\). There exists a constant \(C>0\), independent of \(R\), such that for any function \(h\) of the form (\ref{hinfsumgeq3}) satisfying \(\norm{h}_{**}<\infty\), the Dirichlet problem
    \begin{equation}
    \label{Ldirproblem}
        \left\{\begin{array}{cl}
    L[\phi]=h,& \text{ in }B_{R}, \\
    \quad\:\phi=0,& \text{ on }\partial B_{R},
    \end{array}\right.
    \end{equation}
    has a unique solution \(\phi=iW\psi\) such that
    \begin{equation}
    \label{Ldirproblemest}
       \sup_{B_{R}}\:\abs{\psi}\leq C\norm{h}_{**}. 
    \end{equation}
\end{proposition}

\begin{proof}
    The existence and uniqueness of solutions to (\ref{Ldirproblem}) follows from the results in \cites{mironescu1995,delpinofelmerkowalczyk2004}. Thus it remains to prove estimate (\ref{Ldirproblemest}). If this estimate does not hold, there exists a sequence \(R_{m}\to\infty\) and functions \(h^{m}=iW\big(h^{m}_{1}+ih^{m}_{2}\big)\) of the form (\ref{hinfsumgeq3}), \(\phi^{m}=iW\psi^{m}=iW\big(\psi_{1}^{m}+i\psi_{2}^{m}\big)\) of the form (\ref{psiinfsumgeq3}), such that 
    \begin{equation*}
        \sup_{B_{R_{m}}}\:\abs{\psi^{m}}=1,\quad\norm{h^{m}}_{**}\to0,\text{ as }m\to\infty,
    \end{equation*}and
    \begin{equation*}
        \left\{\begin{array}{cl}
    L[\phi^{m}]=h^{m},& \text{ in }B_{R_{m}}, \\
    \:\:\phi^{m}=0,& \text{ on }\partial B_{R_{m}},
    \end{array}\right.
    \end{equation*}for each integer \(m\geq1\). An application of standard elliptic estimates shows that the functions \(\phi^{m}\) converge (after passing to a subsequence) uniformly on compact subsets of \(\mathbb{R}^{2}\) to a bounded function \(\phi\) satisfying \(L[\phi]=0\) in \(\mathbb{R}^{2}\). Since \(\psi=(iW)^{-1}\phi\) contains only modes \(k\geq3\) in its Fourier expansion, we conclude from Proposition \ref{modegeq2kernel} that necessarily \(\phi=0\). It follows that \(\psi^{m}\to0\) uniformly on compact subsets of \(\mathbb{R}^{2}\backslash\{0\}\). To obtain an outer estimate, we fix a large constant \(\rho>0\) and note, using Proposition \ref{Lpsiouterest}, that we have the inequality 
    \begin{equation*}
        \sup_{B_{R_{m}}\backslash B_{\rho}}\abs{\psi^{m}}\leq C\bigg(\sup_{\partial B_{\rho}}\:\abs{\psi^{m}}+\sup_{B_{R_{m}}\backslash B_{\rho}}\abs{r^{2}h_{1}^{m}}+\sup_{B_{R_{m}}\backslash B_{\rho}}\abs{h_{2}^{m}}\bigg),
    \end{equation*}for all sufficiently large \(m\). Since \(\norm{h^{m}}_{**}\to0\) and \(\psi^{m}\to0\) uniformly on \(\partial B_{\rho}\), we conclude that 
    \begin{equation*}
    \sup_{B_{R_{m}}\backslash B_{\rho}}\abs{\psi^{m}}\to0,\text{ as }m\to\infty.
    \end{equation*}
    
    Combining the observations made so far, we see that the mass of the functions \(\psi^{m}\) must concentrate at the origin as \(m\to\infty\). More precisely, there exists a sequence \(x_{m}\in B_{R_{m}}\) with \(x_{m}\to0\) as \(m\to\infty\) such that \(\abs{\psi^{m}(x_{m})}\geq1/2\) for all \(m\). We now define, for each integer \(m\geq1\), the rescaled function
    \begin{equation*}
        \widetilde{\psi}^{m}(x):=\psi^{m}(\abs{x_{m}}x),\quad x\in B_{R_{m}/\abs{x_{m}}}.
    \end{equation*}After passing to a subsequence, the functions \(\widetilde{\psi}^{m}\) converge uniformly on compact subsets of \(\mathbb{R}^{2}\backslash\{0\}\) to a bounded function \(\widetilde{\psi}\) satisfying
    \begin{equation}
    \label{psitildelimiteq}
        \Delta \widetilde{\psi}+\frac{2}{r}\partial_{r}\widetilde{\psi}+\frac{2i}{r^{2}}\partial_{\theta}\widetilde{\psi}=0,\quad\text{in }\mathbb{R}^{2}\backslash\{0\}.
    \end{equation}Moreover, the limit function \(\widetilde{\psi}\) is not identically zero and admits an expansion of the form (\ref{psiinfsumgeq3}). Decomposing problem (\ref{psitildelimiteq}) in Fourier modes leads to the ODEs (\ref{fmodegeq2systemhomrto0}), which have no bounded non-trivial solutions for \(k\geq3\). This contradicts the fact that \(\widetilde{\psi}\) is non-trivial. The proof of the proposition is thus complete.
    \end{proof}

    We can now complete the proof of the main result of this section.

    \begin{proof}[Proof of Proposition \ref{fouriersumprop}]
        Given \(h\) of the form (\ref{hinfsumgeq3}) satisfying \(\norm{h}_{**}<\infty\), let \(\phi_{R}=iW\psi_{R}\) denote the unique solution of (\ref{Ldirproblem}) on \(B_{R}\) with right-hand side \(h\). Using (\ref{Ldirproblemest}) and standard elliptic estimates, we can extract a subsequence \(R_{m}\to\infty\) such that \(\phi_{R_{m}}\to\phi\) in \(L^{\infty}_{loc}(\mathbb{R}^{2})\), where \(\phi\) is a bounded solution of \(L[\phi]=h\) in \(\mathbb{R}^{2}\). Moreover, \(\phi=iW\psi\) satisfies estimate (\ref{infsumest}) and admits an expansion of the form (\ref{psiinfsumgeq3}). For each \(k\geq3\), the Fourier modes of \(\phi\) are related to those of \(h\) through the equations \(L[iW\psi_{k}^{1}]=iWh_{k}^{1}\) and \(L[iW\psi_{k}^{2}]=iWh_{k}^{2}\). By the uniqueness of bounded solutions to these problems, we conclude that \(\psi_{k}^{1}\) and \(\psi_{k}^{2}\) must agree with the functions constructed in Proposition \ref{fouriermodegeq2prop} and Corollary \ref{fouriermodegeq2cor}. Thus the pair \((\psi_{k1}^{1},\psi_{k2}^{1})\) can be represented through formula (\ref{modegeq2partsol}) with right-hand side \(\widetilde{h}=(h_{k1}^{1},h_{k2}^{1})\), and similarly for \((\psi_{k1}^{2},-\psi_{k2}^{2})\).
    \end{proof}

\subsection*{Acknowledgments}
	M.~del Pino has been supported by the Royal Society Research Professorship grant RP-R1-180114 and by  the  ERC/UKRI Horizon Europe grant  ASYMEVOL, EP/Z000394/1. R.~Juneman has been supported by a University Research Studentship at the University of Bath.

\bibliography{SolvabilityGL}

\end{document}